\documentclass[12pt,a4paper, reqno]{article}

\usepackage{indentfirst} 
\usepackage{amsmath, amsthm, amssymb,xfrac, mathrsfs} 
\usepackage{tikz} 
\usetikzlibrary{hobby}
\usetikzlibrary{decorations.pathreplacing}
\usepackage{wrapfig} 
\usepackage{xcolor} 
\usepackage{verbatim} 
\usepackage{enumerate} 
\usepackage{fullpage}

\usepackage{hyperref} 
\hypersetup{
colorlinks=true,
citecolor=black!50!red,
linkcolor=black!50!red,
linktoc=all
}
\usepackage[all]{hypcap} 

\newcounter{constant}
\newcommand{\newconstant}[1]{\refstepcounter{constant}\label{#1}}
\newcommand{\useconstant}[1]{c_{\textnormal{\tiny \ref{#1}}}}
\newcounter{bigconstant}

\newtheorem{teo}{Theorem}[section]
\newtheorem{prop}[teo]{Proposition}
\newtheorem{lemma}[teo]{Lemma}

\newtheorem{cor}[teo]{Corollary}

\numberwithin{equation}{section} 

\theoremstyle{definition}

\newtheorem{remark}[teo]{Remark}

\renewcommand{\P}{\mathbb{P}}
\newcommand{\E}{\mathbb{E}}
\newcommand{\R}{\mathbb{R}}

\newcommand{\N}{\mathbb{N}}
\newcommand{\Z}{\mathbb{Z}}

\newcommand{\charf}[1]{\mathbf{1}_{#1}}
\newcommand{\roof}[1]{\left\lceil #1 \right\rceil}

\DeclareMathOperator{\dist}{d}

\DeclareMathOperator{\sign}{Sign}
\DeclareMathOperator{\poisson}{Poisson}

\DeclareMathOperator{\expo}{Exponential}
\DeclareMathOperator{\geo}{Geometric}

\title{Local survival of spread of infection among biased random walks}

\author{Rangel Baldasso\footnote{Email: \ r.baldasso@math.leidenuniv.nl; \ Mathematical Institute, Leiden University, P.O. Box 9512, 2300 RA Leiden, The Netherlands.} 
\and Alexandre Stauffer\footnote{Email: \ astauffer@mat.uniroma3.it; \ Universit\`a Roma Tre, Dip.\ di Matematica e Fisica, Largo S.\ Murialdo 1, 00146, Rome, Italy; University of Bath, Dept of Mathematical Sciences, BA2 7AY Bath, UK, supported by EPSRC Fellowship EP/N004566/1.}}

\begin{document}

\maketitle

\begin{abstract}
We study infection spread among biased random walks on $\Z^{d}$. The random walks move independently and an infected particle is placed at the origin at time zero. Infection spreads instantaneously when particles share the same site and there is no recovery. If the initial density of particles is small enough, the infected cloud travels in the direction of the bias of the random walks, implying that the infection does not survive locally. When the density is large, the infection  spreads to the whole $\Z^{d}$. The proofs rely on two different techniques. For the small density case, we use a description of the infected cloud through genealogical paths, while the large density case relies on a renormalization scheme.
\end{abstract}

\section{Introduction}
~
\par We consider here an infection process that evolves on the $d$-dimensional integer lattice, where each individual performs a biased nearest-neighbor random walk. Let $p(\cdot)$ be a nearest-neighbor probability distribution on $\Z^{d}$. We assume without loss of generality, that, for each $i \in [d]$, $0 < p(e_{i}) \leq p(-e_{i}) < 1$, where $\{e_{i}\}_{i =1 }^{d}$ is the canonical basis of $\Z^{d}$, and write $\vec{v}=\sum_{x \sim 0} p(x)\vec{x}$ for the $d$-dimensional drift of the distribution $p(\cdot)$.

\par We consider a ``Poissonian cloud'' of independent continuous-time random walks with jump distribution $p(\cdot)$. More formally, at time zero, each site $x \in \Z^{d}$ receives an i.i.d.\ number of particles $\eta_{0}(x) \sim \poisson(\rho)$, where $\rho>0$ is a given parameter which we call density. Then, each particle evolves as an independent continuous-time random walk that jumps with rate one and whose increments have distribution $p(\cdot)$. Denote by $\eta_{t}(x)$ the number of particles at position $x$ at time $t$.

\par We now define the infection process we will consider. Particles will be of two types: either healthy or infected. At time zero, we add an additional particle at the origin that is declared infected. A given particle is infected at time $t$ if there exists some time $s \in [0,t]$ such that it shared a site with a previously infected particle. Alternatively, one might say that a healthy particle becomes immediately infected if it shares a site with an already infected particle.

\par Let $\xi_{t}(x)$ denote the number of infected particles at $x \in \Z^{d}$ at time $t$. Notice that, at any given time $t$ and site $x \in \Z^{d}$, either all particles at site $x$ are healthy or all are infected.

\par Our interest lies in understanding the behavior of the process $\xi=(\xi_{t})_{t \geq 0}$. We will focus on directions where the drift is negative, more precisely, we assume that $p(e_{1})<p(-e_{1})$ and examine the projection of the infected cloud in this direction\footnote{This is not a restrictive assumption. In fact all our results remain valid if one assumes that $v_{i} = p(e_{i})-p(-e_{i}) \neq 0$, for some $i \in [d]$. In order to simplify the statements, we assume without loss of generality that $v_{1}<0$. This can be achieved, for example, by permuting the coordinates.}. There are two forces here that are in opposition to each other: while each individual particle has a drift away from the origin, in order for the infection to travel towards the same direction, it is necessary that each site is visited only finitely many times by infected particles, which puts a strain on the system. The strength of this opposing force, though, is controlled by the density parameter $\rho$. This hints at the existence of different behaviors for the model, as the density parameter $\rho$ changes.

\newconstant{c:large_deviation}

\par We first consider the case when the density is small. Here, the time it takes for each infected location to be emptied is not enough to sustain the infection process locally, and the infection cloud travels towards $\vec{v}$. To precisely detect this phenomena, consider the quantity
\begin{equation}
r_{t} = \sup \{ \langle x, e_{1} \rangle : \xi_{t}(x)>0\},
\end{equation}
the maximum displacement ``towards the right'' in the $e_{1}$ direction.

\begin{teo}\label{t:small_density}
Assume that $p(e_{1}) < p(-e_{1})$ and set $v_{1}= p(e_{1})-p(-e_{1})$. For any $\delta>0$, there exists a positive density $\rho_{-}>0$ such that, for all $\rho \in (0, \rho_{-})$, there exists a positive constant $\useconstant{c:large_deviation}$ such that, for all $t \geq 0$,
\begin{equation}
\P_{\rho}\left[r_{t} \geq (v_{1}+\delta) t \right] \leq \useconstant{c:large_deviation}e^{-\frac{t}{\useconstant{c:large_deviation}}}.
\end{equation}
\end{teo}

\par As a corollary of the theorem above, we conclude that, provided the density $\rho$ is sufficiently small, every finite subset is eventually free of the infection, since it travels towards negative directions with positive speed.
\begin{cor}\label{cor:small_density}
If $p(e_{1}) < p(-e_{1})$ and $\rho$ is small enough, then each fixed site is visited finitely many times by infected particles, almost surely.
\end{cor}

\newconstant{c:travel_slow}
\newconstant{c:travel_slow_2}

\par When the density is large, the picture is different from Corollary~\ref{cor:small_density}. Even though each particle travels towards the direction $\vec{v}$, the time it takes for all particles in a given site to move allows the infection to actually spread in the opposite direction. Our next theorem states that, provided the density is large enough, $r_{t}$ actually grows with positive speed with large probability.
\begin{teo}\label{t:large_density}
Given $p(\cdot)$ such that $p(e_{1})>0$, there exist a positive constant $\Delta= \Delta(p(\cdot))>0$ and density $\rho_{+} = \rho_{+}(p(\cdot))>0$ such that, for all $\rho > \rho_{+}$, there exist positive constants $\useconstant{c:travel_slow} = \useconstant{c:travel_slow}(p(\cdot), \rho)$, $\useconstant{c:travel_slow_2}= \useconstant{c:travel_slow_2}(p(\cdot), \rho)$, and $t_{0}=t_{0}(p(\cdot), \rho)$ such that
\begin{equation}\label{eq:travels_slow}
\P_{\rho}\left[r_{t} \leq \useconstant{c:travel_slow_2} t\right] \leq e^{-\useconstant{c:travel_slow}\left(\log t\right)^{1+\Delta}}, \quad \text{for all } t \geq t_{0}.
\end{equation}
\end{teo}

\par According to Theorems~\ref{t:small_density} and~\ref{t:large_density}, we have $r_{t} \leq -ct$, if $\rho<\rho_{-}$, and $r_{t} \geq \tilde{c}t$, for $\rho>\rho_{+}$. An interesting open problem would be to prove whether there is a critical point $\rho_{c}$ such that, for $\rho < \rho_{c}$, the infection travels to the left, and, for $\rho > \rho_{c}$, the infection travels to the right.

\par Although the heuristics for the theorems above can be easily understood, turning them into proofs is not straightforward. Verifying such statements for similar models, such as in Kesten and Sidoravicius~\cite{ks}, relies on developing intricate renormalization schemes. In Baldasso and Teixeira~\cite{bt}, the authors consider another model, where particles evolve as a one-dimensional zero-range process without drift and also rely on multiscale renormalization to derive their results.

\bigskip

\textbf{Proof overview.} The proof of Theorem~\ref{t:small_density} follows by path-counting arguments. This proof has two distinct parts that rely on the fact that every site that is infected at time $t$ can be reached by a concatenation of trajectories of random walks, with the first one starting at the origin. We will call such a concatenation of random walk trajectories as a \emph{genealogical path}, but defer the formal definition to Section~\ref{sec:model}. The overall goal of the proof is to show that, for all times $t$, there exists no genealogical path that could bring the infection far from $t \vec{v}$.

\par  The first part of the proof is similar to the analogous statement from Kesten and Sidoravicius~\cite{ks}, that consider the case when the random walks are balanced. In this part, we simply prove that, with sufficiently high probability, there exists no genealogical path starting at the origin that performs too many jumps. This will allow us to reduce the number of genealogical paths we need to consider in the rest of the proof.

\par The second part requires a more refined analysis, as we need to account for the contribution that each particle in a genealogical path could have to ``move'' the infection away from the drift. For this, we first verify that each particle that is followed by a genealogical path gives a contribution towards this displacement away from the drift that has good concentration. From this, we infer that, in order for a genealogical path to move too much away from the drift, it needs to follow  too many different particles. 

\par Theorem~\ref{t:large_density} has a much more intricate proof. Here we rely on multiscale renormalization by considering events where the infection does not travel fast enough in the positive direction. We prove that, provided the density is large enough, the probability of such events is small, by relating events of different scales. A central piece of the proof that allows us to establish such relation is the decoupling for biased random walks.

\bigskip

\textbf{Decoupling.} A decoupling is an estimate on the correlation decay of functions of the space-time configurations with supports that are sufficiently far away. They are powerful tools that replace the use of mixing properties for models that lack such estimates. We here prove a decoupling for the particle system composed of independent continuous-time biased random walks.

\par We regard the collection of space-time configurations as a subset of $\N_{0}^{\Z^{d} \times \R}$, and consider monotone functions of such configurations, which we assume are defined in this larger space.

\par We say that a function $f:\N_{0}^{\Z^{d} \times \R} \to \R$ has support on the space-time box $B \subset \Z^{d} \times \R$ if
\begin{equation}
\eta_{t}(x)=\bar{\eta}_{t}(x), \text{ for all } (x,t) \in B \text{ implies } f(\eta)=f(\bar{\eta}).
\end{equation}
The function $f$ is said decreasing if $\eta \preceq \bar{\eta}$\footnote{Given two space time configurations $\eta$ and $\bar{\eta}$, we say that $\eta \preceq \bar{\eta}$ if $\eta_{t}(x) \leq \bar{\eta}_{t}(x)$, for all $x \in \Z$ and $ t \in \R$.} implies $f(\eta) \geq f(\bar{\eta})$.

\par Given two space-time supports $B_{1}$ and $B_{2}$, their time distance is the quantity
\begin{equation}
\dist_{V}(B_{1}, B_{2})= \inf \left\{|t-s|: \begin{array}{c} \text{ there exist } x, y \in \Z^{d} \text{ such that } \\ (x,t) \in B_{1} \text{ and } (y,s) \in B_{2} \end{array} \right\}.
\end{equation}

\newconstant{c:decoupling}
\newconstant{c:decoupling_2}

\par We will prove a correlation estimate for decreasing functions of the space time with bounded supports that have sufficiently large time distances.
\begin{teo}\label{teo:decoupling}
There exist positive constants $\useconstant{c:decoupling}$ and $\useconstant{c:decoupling_2}$ such that the following holds. Let $B_{1}$ and $B_{2}$ denote two space-time cubes of side-length $n>0$ satisfying
\begin{equation}
\dist_{V}=\dist_{V}(B_{1}, B_{2}) \geq \useconstant{c:decoupling_2},
\end{equation}
and assume, without loss of generality that the time coordinates in the box $B_{1}$ are smaller than the ones in $B_{2}$. For any two decreasing functions of the space-time configurations $f_{1},f_{2}:\N_{0}^{\Z^{d} \times \R} \to [0,1]$ with respective supports in $B_{1}$ and $B_{2}$, we have, for any $\rho \geq 1$,
\begin{equation}\label{eq:decoupling_estimate}
\E_{\rho^{*}}[f_{1}f_{2}] \leq \E_{\rho^{*}}[f_{1}]\E_{\rho}[f_{2}]+\rho^{2d+2}(n+\dist_{T})^{d+1}e^{-\useconstant{c:decoupling}\dist_{T}^{\frac{1}{12}}},
\end{equation}
where $\rho^{*} = \rho\Big(1+\dist_{T}^{-\frac{d}{8(d+2)}}\Big)$.
\end{teo}

\begin{remark}
Notice that the estimate above is not a correlation estimate, since it relates expectations for different density parameters $\rho$ and $\rho^{*}$. A natural question is whether it is possible to obtain such estimates without the use of the so-called sprinkling in the density. This is not the case, as verified in~\cite{rwrw} for a very similar model, composed of discrete-time balanced random walks on $\Z$. In this case, they exhibit an example where correlations decay as $\dist_{T}^{-\frac{1}{2}}$ (see Equation (2.11) from~\cite{rwrw}).
\end{remark}

\begin{remark}\label{remark:extension_decoupling}
We remark that the statement of the theorem above can be extended to allow for the case when the function $f_{1}$ depends not only on the configuration inside the box $B_{1}$, but actually on the whole past of the process up to the upper time limit given by the ball $B_{1}$. 
\end{remark}

\bigskip

\textbf{Proof overview of the decoupling.} The proof of Theorem~\ref{teo:decoupling} relies on a construction of a coupling between two systems $\eta$ and $\eta^{*}$ with respective densities $\rho$ and $\rho^{*}$ such that, for a given subset of $H \in \Z^{d}$, we have $\eta_{t}(x) \leq \eta^{*}_{t}(x)$, for all $x \in H$ with large probability, provided $t$ is large enough. With this coupling in hands, Theorem~\ref{teo:decoupling} follows easily.

\par The construction of the coupling is more intricate. Its nature resembles that of Baldasso and Teixeira~\cite{bt2, bt}. We start with two independent collections of particles $\eta_{0}$ and $\eta_{0}^{*}$. Inside a large set containing $H$, we match particles of $\eta_{0}$ to particles of $\eta^{*}_{0}$. We then let the random walks evolve. Whenever a pair of matched particles meet, they evolve together. Standard heat kernel estimates provide the bounds on the probability of this happening before some given time $s$. As those bounds are not strong enough for the estimates we need, we do a rematching of the particles a polynomial number of times to boost the aforementioned estimates to yield our desired stretched exponential bounds. We remark that this coupling is more robust than that of~\cite{ gs1, psss, pt}, since it does not require such a refined control of heat-kernel estimates.

\bigskip

\textbf{Related works.} There are many different works that treat models for infection spread. Perhaps the most similar to ours is considered by Kesten and Sidoravicius~\cite{ks}. In their case, particles evolve as continuous-time \emph{unbiased} simple random walks and all particles placed initially at the origin begin infected. Once again, infection spreads through contact and there is no recovery. They consider the set $V(t)$ of sites visited by an infected particle up to time $t$ and prove that there exist positive constants $C_{1}$ and $C_{2}$ such that, with large probability, $B(C_{1}t) \subset V(t) \subset B(C_{2}t)$, where $B(k)=[-k,k]^{d}$. In~\cite{ks2}, Kesten and Sidoravicius strenghen the results of~\cite{ks} and conclude that the set $V(t)$ satisfies a shape theorem, while in~\cite{ks3} they studied the case with recovery.

\par The proof in~\cite{ks} shares some similarities with ours. The upper bound is also obtained through path-counting arguments while the lower bound revolves around the construction of a delicate renormalization structure. As we mentioned in the proof overview of our Theorems~\ref{t:small_density} and~\ref{t:large_density}, the first part of our proof follows the path-counting argument for the upper bound in~\cite{ks}, but we need to proceed one step further to control that the infection does not move too much away from the bias of the random walks. With regard to the lower bound, we focus this discussion on the one-dimensional case to highlight the main differences between our proof and that in~\cite{ks}. First,~\cite{ks} observes that the infection front (say, the rightmost infected particle) behaves as a symmetric random walk when there is only one infected particle at the front, whereas the front has a drift to the right when there is more than one particle. This implies that, in order to prove that the infection grows linearly, it suffices to prove that the infection front has at least two particles a positive fraction of time, which they obtain via a renormalization scheme. In our case, where particles have a drift to the left, this strategy fails precisely because  of two reasons. First, having two particles at the front may not be enough to overcome the drift to the left of the random walks, so one needs a sufficiently large number of particles. Second, even if two particles were enough to overcome the drift, just having a positive density of times with two particles at the front may not be enough to compensate for the drift to the left that the front undertake when it has just one particle. Our strategy is then to develop a multiscale renormalization scheme different from that of~\cite{ks}, with a target of controlling instances where the infection does not travel with a minimal positive speed to the right, and prove that events of this form have very small probability.

\par Regarding other works, Gracar and Stauffer~\cite{gs1, gs2} analyzed a more general situation where the random walks move on top of the random conductance model. They prove the existence of a percolation structure (which they call Lipschitz surface) and use this to conclude that the infection spreads with positive speed for $d \geq 2$. A less structured percolating argument was obtained by Stauffer~\cite{stauffer} in continuous space, where particles move as independent Brownian motions.

\par A simpler model that can be viewed as an infection process is the so-called frog model. Here, infected particles perform discrete-time simple random walks, while healthy particles do not move until an infected particle jumps onto their position. A thorough discussion about this model can be found in the survey paper by Popov~\cite{popov}. We just remark that, under some minor conditions on the initial location of the particles, Alves, Machado, and Popov~\cite{amp}, and, independently, Ram\'{i}rez and Sidoravicius~\cite{rs}, prove a shape theorem similar to the one in~\cite{ks2}. This was further strengthened by Alves, Machado, Popov and Ravishankar~\cite{ampr}.

\par Let us now briefly review models where particles do not move independently. Baldasso and Teixeira~\cite{bt} consider particles that move according to a one-dimensional zero-range process. Under mild conditions that garantee the existence of invariant measures for the process, they provide lower and upper bounds for the speed with which the front of the infection grows. Jara, Moreno, and Ram\'{i}rez~\cite{jmr} consider an infection evolving on top of one-dimensional exclusion process and rely on regeneration arguments to prove a law of large numbers and central limit theorem for the infection front.

\par Regarding decoupling estimates (as in our Theorem~\ref{teo:decoupling}), sprinkling ideas were first introduced in the context of random interlacements by Sznitman~\cite{s} and in the context of independent Brownian motions by Sinclair and Stauffer~\cite{ss} (see also~\cite{psss}). These types of inequalities were used to study several conservative particle systems. Peres, Sinclair, Sousi, and Stauffer~\cite{psss}, Benjamini and Stauffer~\cite{bs}, and Stauffer~\cite{stauffer} considered independent Brownian motions. Hil\'{a}rio, den Hollander, Sidoravicius, dos Santos, and Teixeira~\cite{rwrw} treated discrete-time balanced random walks and built on the strategy from Popov and Teixeira~\cite{pt} to provide decouplings for this system. The random conductance model was considered in~\cite{gs1}, while Baldasso and Teixeira developed a decoupling inequality for the one-dimensional zero-range process~\cite{bt} and the one-dimensional simple exclusion process~\cite{bt2}.

\section{Basic definitions}\label{sec:model}
~
\par Let us now precisely construct the particle systems and infection process we consider.

\par Recall that $p(\cdot)$ denotes a nearest-neighbor probability distribution on $\Z^{d}$ such that, for each $i \in [d]$,
\begin{equation}\label{eq:drift}
0 < p(e_{i}) \leq p(-e_{i}) < 1,
\end{equation}
where $\{e_{i}\}_{i =1 }^{d}$ is the canonical basis of $\Z^{d}$. The vector $\vec{v}=\sum_{x \sim 0} p(x)\vec{x}$ is the $d$-dimensional drift of the distribution $p(\cdot)$. Due to~\eqref{eq:drift}, every coordinate of $\vec{v}$ is non-positive. Furthermore, we assume that $p(e_{1}) < p(-e_{1})$, so that $v_{1}$, the first coordinate of $\vec{v}$, is negative.

\par For each $x \in \Z^{d}$ and $n \in \N$, let $S^{x,n} = \left(S^{x,n}_{t}\right)_{t \geq 0}$ denote an independent copy of a rate-one continuous-time random walk with transition probability $p(\cdot)$ and $S^{x,n}_{0}=x$, for all $n \in \N$. Denote this collection by $\mathcal{S}$.

\par Given a non-negative parameter $\rho \geq 0$, consider, for each $x \in \Z^{d}$, an independent random variable $\eta_{0}(x)$ with distribution $\poisson(\rho)$. For each positive time $t>0$, let
\begin{equation}\label{eq:biased_irw}
\eta_{t}(x) = \sum_{y \in \Z^{d}} \sum_{n=1}^{\infty} \charf{\{S^{y,n}_{t}=x\}}\charf{\{n \leq \eta_{0}(y)\}}
\end{equation}
denote the number of particles at position $x$ at time $t$. We write $\P_{\rho}$ for the distribution of the process $\eta=\left(\eta_{t}\right)_{t \geq 0}$.

We see particles in a given space-time point as ordered in a pile. This ordering can be arbitrarily chosen, and we will use it to talk about the $k$-th particle in a site. Furthermore, notice that it also makes sense to talk about the $k$-th particle that jumps into a site $x$ after time $t$ and that this does not depend on this ordering of particles in each site.

\par The product measure with marginals $\poisson(\rho)$ is invariant for the process, and we call the quantity $\rho$ the density of the system. Besides, if $\rho' \leq \rho$, it is possible to define an order preserving coupling between two processes $\eta^{\rho'}$ and $\eta^{\rho}$ with respective densities $\rho'$ and $\rho$ such that
\begin{equation}
\eta^{\rho'}_{t}(x) \leq \eta^{\rho}_{t}(x),
\end{equation}
for all $t \geq 0$ and $x \in \Z^{d}$: one simply uses the same collection of walks $\mathcal{S}$ to evolve both processes and consider the initial conditions $\eta^{\rho'}_{0}$ and $\eta^{\rho}_{0}$ that satisfy $\eta^{\rho'}_{0}(x) \leq \eta^{\rho}_{0}(x)$, for all $x \in \Z^{d}$.

\bigskip

\par We now proceed to define the infection process $(\xi_{t})_{t \geq 0}$. Recall we add an infected particle at the origin at time zero. At any given time $t \geq 0$, $\xi_{t}(x)$ denotes the number of infected particles at $x \in \Z^{d}$. At time zero, only particles at the origin are infected, which means that $\xi_{0}(x)=(\eta_{0}(0)+1)\charf{\{x=0\}}$ . As for the evolution, each time an infected particle jumps to a site with healthy particles, all particles at that given site become infected. Furthermore, if a healthy particle jumps towards a site with infected particles, it immediately becomes infected. This in particular implies that in every site and non-negative time, either all particles are healthy or all are infected.

\subsection{Genealogical infected paths}\label{subsec:GIP}
~
\par We can see the infection mechanism through genealogical paths. In order to define these paths, we introduce a notation for the trajectory of a particle. If $X$ denotes a particle present at time zero, we denote by $(X(s))_{s \geq 0}$ the path that it performs. We say that a particle $X$ becomes infected at time $t$ if it is healthy before time $t$ and infected after time $t$, i.e., $t$ is the first time it shares a site with another infected particle.

\par We say that $\gamma:[0,t] \to \Z^{d}$ is a genealogical infected path up to time $t$ (GIP($t$)) if $\gamma(0)=0$ and there exist a sequence of times $0 = t_{0} < t_{1} < \dots < t_{n} \leq t$ and a sequence of particles $X_{1}, \dots X_{n+1}$ such that $X_{1}(0)=0$, $X_{i}$ becomes infected at time $t_{i-1}$, and, for all $s \in [t_{i-1}, t_{i}]$ we have $\gamma(s)=X_{i}(s)$, for all $i \in [n]$; for convenience of notation, we assume that $t_{n+1}=t$. Of course, $\xi_{t}(x) >0$ if and only if there exists a GIP($t$) $\gamma$ with $\gamma(t)=x$. See Figure~\ref{fig:GIP} for a representation of a GIP($t$).

\par We will identify a GIP($t$) by the following: the number $n$ of particles it follows, a vector $(k_{1}, \dots, k_{n})$ with non-negative entries that counts the number of jumps each particle performs while it is being followed, and a vector with non-zero integer entries $(j_{0}, j_{1}, \dots j_{n-1})$ that identifies which is the next particle to be followed ($j_{0}$ identifies the first particle that is followed and starts at the origin). If $j_{i}>0$, then when $X_i$ makes its last jump we take $X_{i+1}$ to be the $j_i$-th healthy particle that was present at the site where $X_i$ jumped to (and thus became infected via $X_{i}$). Whenever $j_{i}<0$, we wait $X_i$ to perform all its $k_i$ jumps, and only at this moment wait for $|j_i|$ healthy particles to jump on the site where $X_i$ is (and thus becomes infected for the first time when it meets $X_{i}$), taking the $|j_i|$-th such particle as $X_{i+1}$.

\par This identification has some particularities we need to address. First, observe that, if $k_{i}=0$, we demand that $j_{i}<0$, i.e., after following a particle that does not jump, we need to follow a particle that is not yet in the site we are considering. Moreover, we will denote by $k = \sum_{i=1}^{n}k_{n}$ the total number of jumps of a GIP. Finally, it is not always the case that all possible choices for these vectors will yield a GIP, but all GIP can be obtained by some choice of such values.

\begin{figure}\label{fig:GIP}
\begin{center}
\begin{tikzpicture}

\draw[<->, thick] (-3.5,0)--(3.5,0);

\foreach \x in {-3, -2, -1, 0, 1, 2, 3}{
\draw[->, thick] (\x,0)--(\x,3.8);
}

\draw[blue, line width=2.5](0,0)--(0, 0.6);
\draw[blue, line width=2.5](-1,0.6)--(-1, 1);
\draw[blue, line width=2.5](-2,1)--(-2, 1.5);
\draw[red, line width=2.5](-2,1.5)--(-2, 1.9);
\draw[green!40!black, line width=2.5](-2,1.9)--(-2, 2.4);
\draw[green!40!black, line width=2.5](-1,2.4)--(-1, 2.7);
\draw[green!40!black, line width=2.5](-2,2.7)--(-2, 3.1);
\draw[purple, line width=2.5](-3,3.1)--(-3,3.7);

\draw[->, thick, dashed] (0,0.6)--(-1,0.6);
\draw[->, thick, dashed] (-1,1)--(-2,1);
\draw[->, thick, dashed] (-2,2.4)--(-1,2.4);
\draw[->, thick, dashed] (-1,2.7)--(-2,2.7);
\draw[->, thick, dashed] (-2,3.1)--(-3,3.1);
\end{tikzpicture}
\caption{An example of a genealogical infected path. Different colors stand for different particles followed. Following the representation of the path via the quantities introduced above, we have $n=4$ and $(k_{1}, k_{2}, k_{3},k_{4}) = (2, 0, 3, 0)$. Furthermore, notice that the transition from the first (blue) particle to the second (red) particle happens with an index $j_{1}<0$. Since the second particle does not jump, we have $j_{2}<0$ as well. In the last transition, from a green to a purple particle, the index $j_{3}$ is positive.}
\end{center}
\end{figure}
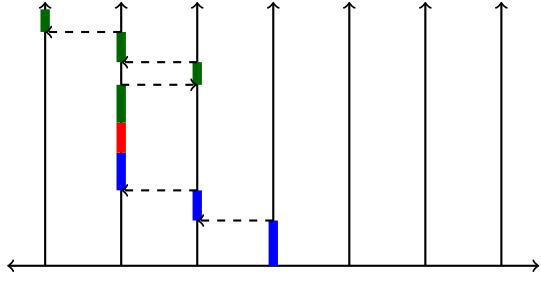

\section{The small density case: proof of Theorem~\ref{t:small_density}}\label{sec:small_density}
~
\par We now consider the case when the density is very small. Using a first moment computation, we will prove that there exists a small density $\rho>0$ such that, with large probability, the infection travels towards the negative direction in the first coordinate axis.

\par We will control how infection spreads by using genealogical paths. The first proposition we prove states that it is unlikely to exist a GIP that jumps many times. This is a general statement that does not depend on the probability distribution $p(\cdot)$. This transition kernel will be important when we consider finer properties of the model.

\newconstant{c:many_jumps}

\begin{prop}\label{prop:many_jumps}
For any $\rho \in (0,1)$, there exists a positive constant $\useconstant{c:many_jumps}=\useconstant{c:many_jumps}(\rho)$, which might be taken to be monotone non-decreasing as a function of $\rho$, such that, for all $t \geq 0$,
\begin{equation}
\P_{\rho}\left[\begin{array}{c}\text{there exists a \emph{GIP($t$)} that} \\ \text{jumps more than } \useconstant{c:many_jumps}t \text{ times}
\end{array}\right] \leq e^{-\useconstant{c:many_jumps}t+1}.
\end{equation}
\end{prop}

\newconstant{c:finite_speed}

\par As a byproduct of the proposition above, we immediately obtain the following result.
\begin{prop}\label{prop:finite_speed}
For any densities $\rho \in (0,1)$, there exists $\useconstant{c:finite_speed}>0$ and $\alpha>0$ such that, for all $t \geq 0$,
\begin{equation}
\P_{\rho}\left[||x||> \alpha t, \text{ for some } x \in \Z^{d} \text{ such that } \xi_{t}(x)>0 \right] \leq \useconstant{c:finite_speed}e^{-\frac{t}{\useconstant{c:finite_speed}}}.
\end{equation}
\end{prop}

\begin{proof}[Proof of Proposition~\ref{prop:many_jumps}]
The proof of this statement relies on a first moment calculation. We will bound the expectation of the number of GIPs that jump more than $\useconstant{c:many_jumps}t$ times before time $t$.

The discussion gets simplified when we use the identification of GIP introduced in Subsection~\ref{subsec:GIP}. A GIP is identified by the number of jumps $k$, the number $n$ of particles it follows, a vector $(k_{1}, \dots, k_{n})$ with non-negative entries that counts the number of jumps each particle performs and a vector with non-zero integer entries $(j_{0}, j_{1}, \dots j_{n-1})$ that identifies which is the next particle to be followed. As in Subsection~\ref{subsec:GIP}, we will denote by $X_{1}, \dots X_{n}$ the collection of particles followed by the GIP.

If $k_{i}>0$, then the time it takes to follow particle $X_{i}$ until its last jump is equal in distribution to the sum of $k_{i}$ independent $\expo(1)$ random variables. We will call these exponential times $\mathcal{T}_{i}$. Besides, if $j_{i}<0$, we gain a time contribution that comes from the fact that, after particle $X_i$ does its $k_i$-th jump, it has to wait until $j_{i}$ healthy particles jump into its site from one of its neighboring sites. We will call such times as $\mathcal{W}_{i,\ell}$, where $\ell$ ranges from $1$ to $|j_{i}|$. Note that the number of healthy particles in a given such neighboring site is distributed according to a Poisson random variable of intensity $\rho$ times the probability that a particle moving as a biased random walk did not touch other infected particles in the past; this follows from the thinning of Poisson point processes. We simply bound this probability by one. Moreover, we still need to account for the possibility that $X_{i}$ jumps, which happens with rate one. Hence, $(\mathcal{W}_{i,\ell})_{\ell}$ is stochastically dominated by a sequence of $|j_{i}|$ independent $\exp(1+\rho)$ random variables. Notice furthermore that the probability that $X_{i}$ jumps before a healthy particle arrives from a neighboring site is at least $\frac{1}{1+\rho}$. In the case when the particle we are following jumps before all the $|j_{i}|$ new particles arrived, we disregard the path. From the above consideration, the probability that the path is not disregarded as described above is at most the probability that $G_{i} > |j_{i}|$, where $ G_{i} \sim \geo\left(\frac{1}{\rho+1}\right)$. From the strong Markov property, it follows that the random variables $\mathcal{T}_{i}$, $\mathcal{W}_{i,\ell}$, and $G_{i}$ are independent.

\par Let $\mathcal{G}_{t}$ denote the number of GIPs that jump more than $\useconstant{c:many_jumps}t$ times before time $t$, where $\useconstant{c:many_jumps}$ is a constant that will be chosen later. The discussion above allows us to bound
\begin{equation}\label{eq:expectation_G}
\begin{split}
\E_{\rho}[\mathcal{G}_{t}] & \leq \E_{\rho}\left[\sum_{k=\lceil \useconstant{c:many_jumps}t \rceil}^{\infty}\sum_{n=1}^{\infty}\sum_{\substack{(k_{1}, \dots k_{n}) \colon \\ \sum_{i=1}^{n}k_{i}=k}} \sum_{(j_{0}, \dots j_{n-1})} \charf{\left\{\sum_{i=1}^{k}\mathcal{T}_{i}+\sum_{i:j_{i}<0}\sum_{\ell=1}^{|j_{i}|}\mathcal{W}_{i, \ell}<t\right\}} \right. \\
& \qquad \qquad \left. \prod_{i=1}^{n}\left(\charf{\left\{\mathcal{N}_{i}>j_{i}\right\}}\charf{\{j_{i} \geq 1\}}+\charf{\left\{G_{i}>|j_{i}|\right\}}\charf{\{j_{i} \leq -1\}}\right) \charf{\{n \leq k+J\}} \right],
\end{split}
\end{equation}
where $\mathcal{N}_{i}$ counts the number of particles present at the site onto which $X_i$ makes its last jump.

Let us now estimate the expectation above. Notice first that the possible choices for the partitions $(k_{1}, \dots k_{n})$ such that $\sum_{i=1}^{n}k_{i} = k$ is bounded by
\begin{equation}\label{eq:bound_vectors_with_given_sum}
\binom{n+k-1}{n-1} \leq 2^{n+k},
\end{equation}
Write $J = \sum_{i: j_{i}<0}|j_{i}|$ and notice that the random variables $\sum_{i=1}^{k}\mathcal{T}_{i}+\sum_{i:j_{i}<0}\sum_{\ell=1}^{|j_{i}|}\mathcal{W}_{i, \ell}$ stochastically dominate a sum of $k+J$ i.i.d.\ $\expo(1+\rho)$ random variables. This allows us to bound
\begin{equation}
\P_{\rho}\left[\sum_{i=1}^{k}\mathcal{T}_{i}+\sum_{i:j_{i}<0}\sum_{\ell=1}^{|j_{i}|}\mathcal{W}_{i, \ell} \leq t\right] \leq \P[X \geq k+J],
\end{equation}
where $X \sim \poisson\big((1+\rho)t\big)$. Notice also that the number of walks we follow $n$ is upper bounded by $k+J$, since, whenever $i \in [n]$ is such that $k_{i}=0$, we have $j_{i} <0$.

To bound the expectation, we first divide the sum according to which subset of indices $A \subset [n-1]$ is such that $j_{i} <0$ for $i \in A$. Notice that, for a fixed choice of set $A$, using that $\mathcal{N}_{i}$ has Poisson distribution with parameter $\rho$,
\begin{equation}\label{eq:bound_theorem}
\sum_{\ell \notin A} \sum_{j_{\ell} \geq 1} \E_{\rho} \left[ \prod_{i \notin A}\charf{\left\{\mathcal{N}_{i}>j_{\ell}\right\}} \prod_{\tilde{\ell} \in A} \charf{\left\{G_{\tilde{\ell}}>|j_{\tilde{\ell}}| \right\}} \right] \leq \left(\frac{\rho}{1+\rho}\right)^{J} \rho^{n-1-|A|}(\rho+1),
\end{equation}
where the term $\rho+1$ comes from the number of particles at the origin at time zero.

From the discussion above we obtain the bound
\begin{equation}\label{eq:GIP_bound}
\E_{\rho}[\mathcal{G}_{t}] \leq \sum_{k=\lceil \useconstant{c:many_jumps}t \rceil}^{\infty}\sum_{n=1}^{\infty}\sum_{A \subset [n-1]} \sum_{j_{i}: i \in A} 2^{n+k} \P[X \geq k+J] \left(\frac{\rho}{1+\rho}\right)^{J} \rho^{n-1-|A|}(\rho+1)\charf{\{n \leq k+J\}}.
\end{equation}
We now observe that that the number of choices of indices $j_{i}$ with $i \in A$ such that $J =\sum_{i \in A}|j_{i}|$ is bounded by $\binom{J-1}{|A|-1}$. This allows us to bound the quantity above by
\begin{equation}
\sum_{k=\lceil \useconstant{c:many_jumps}t \rceil}^{\infty}\sum_{n=1}^{\infty}\sum_{A \subset [n-1]} \sum_{J=|A|}^{\infty}\binom{J-1}{|A|-1} 2^{n+k} \P[X \geq k+J] \left(\frac{\rho}{1+\rho}\right)^{J} \rho^{n-1-|A|}(\rho+1)\charf{\{n \leq k+J\}}.
\end{equation}

Second, we bound the number of choices for $A$ according to its size. This yields the bound
\begin{equation}
\sum_{k=\lceil \useconstant{c:many_jumps}t \rceil}^{\infty} \sum_{n=1}^{\infty}\sum_{\ell=0}^{n-1} \sum_{J=\ell}^{\infty}\binom{n}{\ell}\binom{J}{\ell} 2^{n+k} \P[X \geq k+J] \left(\frac{\rho}{1+\rho}\right)^{J} \rho^{n-1-\ell}(\rho+1)\charf{\{n \leq k+J\}}.
\end{equation}

We now change the order of summations, and conclude that the r.h.s.\ of~\eqref{eq:GIP_bound} is bounded by
\begin{equation}
\begin{split}
\sum_{k=\lceil \useconstant{c:many_jumps}t \rceil}^{\infty} & \sum_{J=1}^{\infty}\sum_{\ell=0}^{J} \sum_{n=\ell+1}^{k+J}\binom{n}{\ell}\binom{J}{\ell} 2^{n+k} \P[X \geq k+J] \left(\frac{\rho}{1+\rho}\right)^{J} \rho^{n-1-\ell}(\rho+1) \\
& \leq \sum_{k= \lceil \useconstant{c:many_jumps}t \rceil }^{\infty}\sum_{J=1}^{\infty} (k+J)8^{k+J}\P[X \geq k+J].
\end{split}
\end{equation}

We now use the estimates from Proposition~\ref{prop:concentration}. By choosing $\useconstant{c:many_jumps} \geq 2(\rho+1)$, we obtain
\begin{equation}
\E_{\rho}[\mathcal{G}_{t}] \leq \sum_{k= \lceil \useconstant{c:many_jumps}t \rceil }^{\infty}\sum_{J=1}^{\infty} (k+J)8^{k+J} \exp{\left\{-\useconstant{c:concentration}(k+J) \log \frac{k+J}{(1+\rho)t}\right\}}.
\end{equation}

To conclude, we set $L=\max\{k,J\} \leq k+J \leq 2L$ in the equation above to obtain that the r.h.s.\ of~\eqref{eq:GIP_bound} is bounded by
\begin{equation}
\begin{split}
\sum_{L=\lceil \useconstant{c:many_jumps}t \rceil}^{\infty} & 4L^{2}8^{2L}\exp{\left\{-\useconstant{c:concentration}L \log \frac{L}{(1+\rho)t}\right\}}\\
& \leq \sum_{L=\lceil \useconstant{c:many_jumps}t \rceil}^{\infty} e^{4L}e^{6L}\exp{\left\{-\useconstant{c:concentration}L \log \frac{L}{(1+\rho)t}\right\}} \\
& \leq \sum_{L=\lceil \useconstant{c:many_jumps}t \rceil}^{\infty} \exp{\left\{-L\left(\useconstant{c:concentration}\log \frac{\useconstant{c:many_jumps}}{(1+\rho)} - 10 \right)\right\}} \\
& \leq e^{-\useconstant{c:many_jumps}t+1},
\end{split}
\end{equation}
if $\useconstant{c:many_jumps}=\useconstant{c:many_jumps}(\rho)=12^{\frac{1}{\useconstant{c:concentration}}}(1+\rho)>0$.

Combining the bound above with~\eqref{eq:GIP_bound}, we obtain that
\begin{equation}
\E_{\rho}[\mathcal{G}_{t}] \leq e^{-\useconstant{c:many_jumps}t+1}.
\end{equation}
To conclude the proof, simply apply Markov's inequality to obtain
\begin{equation}
\P_{\rho}[\mathcal{G}_{t} \geq 1] \leq \E_{\rho}[\mathcal{G}_{t}] \leq e^{-\useconstant{c:many_jumps}t+1}.
\end{equation}
\end{proof}

\newconstant{c:mushroom}

\par We now proceed to conclude the proof of Theorem~\ref{t:small_density}. In view of Proposition~\ref{prop:many_jumps}, we may restrict ourselves to paths that do not jump many times until time $t$. Before presenting the proof, we provide some basic facts about biased random walks that will be used in the proof, since GIPs are constructed by concatenating such objects.
\begin{lemma}\label{lemma:mushroom}
Let $(X_{t})_{t \geq 0}$ be a random walk starting from the origin, and let $\vec{v}$ denote its drift. For any $\epsilon>0$, there exist a positive random variable $\mathcal{R}_{\epsilon}$ and a positive constant $\useconstant{c:mushroom}=\useconstant{c:mushroom}(p(\cdot), d, \epsilon) \in (0,1)$ such that, almost surely
\begin{equation}\label{eq:mushroom_1}
||X_{t}- t\vec{v}|| \leq \max\{\mathcal{R}_{\epsilon}, \epsilon  t\}, \quad \text{ for all } t \geq 0,
\end{equation}
and, for every $u \geq 0$,
\begin{equation}\label{eq:mushrrom_2}
\P\left[\mathcal{R}_{\epsilon} \geq u\right] \leq \useconstant{c:mushroom}^{-1}e^{-\useconstant{c:mushroom}u}.
\end{equation}
In particular, $\mathcal{R}_{\epsilon}$ is stochastically dominated by $\frac{1}{\useconstant{c:mushroom}}Y-\frac{\log \useconstant{c:mushroom}}{\useconstant{c:mushroom}}$, where $Y \sim \expo(1)$.
\end{lemma}

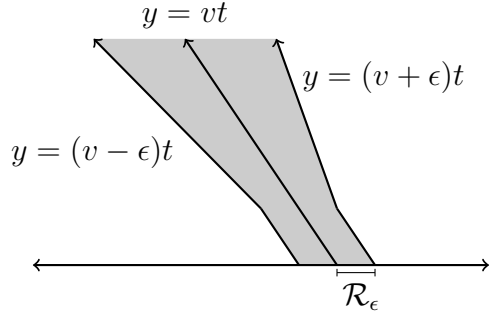
\begin{figure}
\begin{center}
\begin{tikzpicture}

\fill[black!20!white](-0.5,0)--(0.5,0)--(0,0.75)--(-0.8,3)--(-3.2,3)--(-1,0.75);
\draw[<->, thick](-4,0)--(2,0);
\draw[->, thick](0,0)--(-2,3);
\draw[thick](0.5,0)--(0,0.75);
\draw[thick](-0.5,0)--(-1,0.75);
\draw[->, thick](0,0.75)--(-0.8,3);
\draw[->, thick](-1,0.75)--(-3.2,3);

\draw (0,-0.05)--(0,-0.15);
\draw (0.5,-0.05)--(0.5,-0.15);
\draw (0,-0.1)--(0.5,-0.1);

\node[right] at (-0.6,2.5){$y=(v+\epsilon)t$};
\node[left] at (-2,1.5){$y=(v-\epsilon)t$};
\node[above] at (-2,3){$y=vt$};

\node[below] at (0.32,-0.1){$\mathcal{R}_{\epsilon}$};
\end{tikzpicture}
\caption{The random variable $\mathcal{R}_{\epsilon}$ in dimension one. Notice that the random walk is completely contained in the gray area.}
\end{center}
\end{figure}

\begin{proof}
Begin by using Lemma~\ref{lemma:rw_deviation} to bound, for $T \geq \frac{2}{\epsilon}$,
\begin{equation}\label{eq:mushroom_3}
\begin{split}
\P & \left[||X_{t}- t \vec{v}|| \geq \epsilon t, \text{ for some } t \geq T \right] \\
& \qquad \qquad \leq \sum_{k=1}^{\infty}\P\left[||X_{t}- t \vec{v}|| \geq \epsilon kT, \text{ for some } t \in [kT, (k+1)T] \right] \\
& \qquad \qquad \leq \sum_{k=1}^{\infty}\P\left[||X_{t}- t \vec{v}|| \geq \frac{\epsilon}{2} (k+1)T, \text{ for some } t \in [kT, (k+1)T] \right] \\
& \qquad \qquad \leq \sum_{k=1}^{\infty}\useconstant{c:rw_deviation}^{-1}e^{-\useconstant{c:rw_deviation}(k+1)T} \leq c^{-1}e^{-cT},
\end{split}
\end{equation}
for some suitable positive constant $c>0$.

Due to Borel-Cantelli Lemma, the random variable
\begin{equation}
\bar{\mathcal{R}}_{\epsilon}=\inf\left\{u \in \left[\frac{2}{\epsilon}, \infty\right) \cap \N: ||X_{t}- t \vec{v}|| \leq \epsilon t, \text{ for all } t \geq u \right\}
\end{equation}
is almost surely finite. Define now
\begin{equation}
\mathcal{R}_{\epsilon} = \sup\Big\{||X_{t}- t \vec{v}||: t \in [0, \bar{\mathcal{R}}_{\epsilon}]\Big\},
\end{equation}
so that~\eqref{eq:mushroom_1} clearly holds. Finally, observe that, as an immediate consequence of~\eqref{eq:mushroom_3} and Lemma~\ref{lemma:rw_deviation}, one obtains that, for all $u \geq 0$,
\begin{equation}
\begin{split}
\P\left[\mathcal{R}_{\epsilon} \geq u\right] & \leq \P\left[\bar{\mathcal{R}}_{\epsilon} \geq u\right] +\P\left[\sup\left\{||X_{t}- t \vec{v}||: t \in [0, \bar{\mathcal{R}}_{\epsilon}]\right\} \geq u, \bar{\mathcal{R}}_{\epsilon} \leq u \right] \\
& \leq c^{-1}e^{-cu}+\P\left[\sup\left\{||X_{t}- t \vec{v}||: t \in [0, u]\right\} \geq u\right] \\
& \leq \useconstant{c:mushroom}^{-1}e^{-\useconstant{c:mushroom}u},
\end{split}
\end{equation}
concluding the proof.
\end{proof}

\begin{remark}
Note that the whole argument in the proof above crucially relies on the fact that in a GIP we start following a given particle only at the very moment when it gets infected. This implies that a particle is followed at most once, and a particle that we start to follow at some time $t$ has never intersected the GIP before time $t$, reducing dependences.
\end{remark}

\par We are now in position to prove Theorem~\ref{t:small_density}. The proof is also based on the first moment method by bounding the expected number of GIPs that do not behave as expected. In view of Proposition~\ref{prop:many_jumps}, we may consider only GIPs that do not jump many times.
\begin{proof}[Proof of Thereom~\ref{t:small_density}]
We assume that $\rho < \tfrac{1}{3}$. By possibly increasing the value of $\useconstant{c:large_deviation}$, we can assume $t \geq 1$. Furthermore, by monotonicity, we may assume that $\delta < \frac{1}{2}|v_{1}|$ and denote by $\mathcal{H}_{t}$ the number of GIPs that jump at most $\useconstant{c:many_jumps}t$ times and are such that their endpoint $x$ satisfy $\langle x-t\vec{v}, e_{1} \rangle \geq \delta t$, where $\useconstant{c:many_jumps}$ is the constant from Proposition~\ref{prop:many_jumps}.

\newconstant{c:poisson}

We first apply Lemma~\ref{lemma:mushroom}. Using the fact that $\delta < |v_{1}|$ and that $v_{1}<0$, the maximum displacement of a given particle towards the positive direction in the axis $e_{1}$ can be bounded by
\begin{equation}
\begin{split}
\sup_{s} \, \langle X_{s} , e_{1} \rangle & \ = \sup_{s} \left\{ \langle X_{s}- \vec{v}s , e_{1} \rangle +v_{1}s \right\} \\
& \leq \max \left\{\sup_{s}\left\{v_{1}s+ \delta s \right\}, \sup_{s}\left\{ v_{1}s + \mathcal{R}_{\delta}\right\} \right\} \\
& \leq \mathcal{R}_{\delta},
\end{split}
\end{equation}
since $v_{1}<0$ and $v_{1}+\delta<0$. In particular, the maximum displacement towards the positive direction of a GIP can be bounded by a sum of i.i.d.\ random variables with distribution $\mathcal{R}_{\delta}$, one for each followed particle in the path. Notice that independence of the random variables comes from the fact that we only follow newly infected particles. In particular, if the path follows at most $\alpha t$ particles, its maximum displacement is bounded by a sum of $\alpha t$ i.i.d.\ random variables $(\mathcal{R}_{\delta}^{i})_{i=1}^{\alpha t}$ with distribution $\mathcal{R}_{\delta}$. This implies that the probability that a given fixed GIP\footnote{Here we mean that a family of parameters $k$, $n$, $(k_{1}, \dots, k_{n})$, and $(j_{0}, \dots j_{n-1}$ is fixed and consider the GIP associated to it, if it exists.} has displacement to the right bigger than $\delta t$ is bounded by
\begin{equation}
\P\left[\sum_{j=1}^{\alpha t} \mathcal{R}_{\delta}^{j} \geq \delta t\right] \leq  \P\left[\sum_{j=1}^{\alpha t} Y_{j} \geq t\useconstant{c:mushroom}\delta +\alpha t\log \useconstant{c:mushroom}\right] = \P\left[Z \leq \alpha t\right],
\end{equation}
where $Y_{j}$ are i.i.d.\ $\expo(1)$ random variables and $Z \sim \poisson \left( t\useconstant{c:mushroom}\delta +\alpha t\log \useconstant{c:mushroom} \right)$. Choose now $\alpha$ small enough such that $\useconstant{c:mushroom}\delta +\alpha \log \useconstant{c:mushroom} > \alpha$ and observe that there exists a positive constant $\useconstant{c:poisson}=\useconstant{c:poisson}(\alpha, \delta, \useconstant{c:mushroom})$ such that
\begin{equation}\label{eq:deviation_mushrooms}
\P\left[\sum_{j=1}^{\alpha t} \mathcal{R}_{\delta}^{j} \geq \delta t\right] \leq \P\left[Z \leq \alpha t\right] \leq e^{-\useconstant{c:poisson}t}.
\end{equation}

We are now in position to bound the expectation of $\mathcal{H}_{t}$. Reasoning similarly as in~\eqref{eq:expectation_G} (and the paragraph preceding this equation), with the same notation, we can obtain the bound
\begin{equation}
\begin{split}
\E_{\rho}[\mathcal{H}_{t}] & \leq \E_{\rho}\left[\sum_{n=1}^{\infty}\sum_{(k_{i})_{i=1}^{n}: \sum_{i=1}^{n} k_{i} \leq \useconstant{c:many_jumps}t}\sum_{(j_{0}, \dots j_{n-1})} \charf{\left\{\sum_{i=1}^{n}\mathcal{R}_{\delta}^{i} \geq \delta t\right\}} \right. \\
& \qquad \qquad \qquad \qquad \left. \prod_{i=0}^{n-1}\left(\charf{\left\{\mathcal{N}_{i}>j_{i}\right\}}\charf{\{j_{i} \geq 1\}}+\charf{\left\{G_{i}>|j_{i}|\right\}}\charf{\{j_{i} \leq -1\}}\right) \right],
\end{split}
\end{equation} 
where $n$ denotes the number of particles followed in a given path, the vector $(k_{i})_{i=1}^{n}$ counts how many jumps each of the particles performs and $(j_{i})_{i=0}^{n-1}$ controls transitions between particles.

Proceeding as in the proof of Proposition~\ref{prop:many_jumps} (in particular, Equation~\eqref{eq:bound_theorem}) we obtain
\begin{equation}
\E_{\rho}\left[ \sum_{(j_{0}, \dots j_{n-1})} \prod_{i=0}^{n-1}\left(\charf{\left\{\mathcal{N}_{i}>j_{i}\right\}}\charf{\{j_{i} \geq 1\}}+\charf{\left\{G_{i}>|j_{i}|\right\}}\charf{\{j_{i} \leq -1\}}\right) \right] \leq (3\rho)^{n-1}\frac{(1+\rho)^{2}}{1-\rho} \leq 6(3\rho)^{n-1}.
\end{equation}

Combining the above with the bound in~\eqref{eq:deviation_mushrooms} yields
\begin{equation}\label{eq:expectation_H_2}
\E_{\rho}[\mathcal{H}_{t}] \leq 6\sum_{n=1}^{\infty} \sum_{(k_{i})_{i=1}^{n}: \sum_{i=1}^{n} k_{i} \leq \useconstant{c:many_jumps}t}\left(\charf{\left\{n \geq \alpha t\right\}}+\charf{\left\{n \leq \alpha t\right\}}e^{-\useconstant{c:poisson}t}\right)(3\rho)^{n-1}.
\end{equation}

To estimate the number of vectors $(k_{i})_{i=1}^{n}$ such that $\sum_{i=1}^{n}k_{i} \leq \useconstant{c:many_jumps}t$, we bound this quantity by the number of vectors $(k_{i})_{i=1}^{n+1}$ such that $\sum_{i=1}^{n+1}k_{i} = \useconstant{c:many_jumps}t$ and apply the bound in~\eqref{eq:bound_vectors_with_given_sum}. We now combine this bound with~\eqref{eq:expectation_H_2} to obtain, for $\alpha \leq \useconstant{c:many_jumps}$,
\begin{equation}\label{eq:expectation_H_4}
\begin{split}
\E_{\rho}[\mathcal{H}_{t}] & \leq 6\sum_{n=1}^{\infty}\left(\charf{\left\{n \geq \alpha t\right\}}+\charf{\left\{n \leq \alpha t\right\}}e^{-\useconstant{c:poisson}t}\right)(3\rho)^{n-1}\binom{\lfloor\useconstant{c:many_jumps}t\rfloor+n}{n} \\
& \leq 12 \cdot 2^{\useconstant{c:many_jumps}t}\sum_{n \geq \alpha t}(6\rho)^{n-1} + \frac{2}{\rho}e^{-\useconstant{c:poisson}t}  \sum_{n=1}^{\alpha t}\left(\frac{6\rho \useconstant{c:many_jumps} t e}{n}\right)^{n} \\
& \leq 2^{\useconstant{c:many_jumps}t} \frac{(6 \rho)^{\alpha t-1}}{1-6 \rho}+\frac{2}{\rho}\alpha t e^{-\useconstant{c:poisson}t} e^{6 \rho \useconstant{c:many_jumps}t}, 
\end{split}
\end{equation}
where the bound in the last summation is obtained by maximizing the expression $\left(\frac{6\rho \useconstant{c:many_jumps}t e}{n}\right)^{n}$ in $n$.

From~\eqref{eq:expectation_H_4}, one easily obtains that, provided $\rho$ is small enough, there exists a positive constant $\useconstant{c:large_deviation}$ such that
\begin{equation}
\E_{\rho}[\mathcal{H}_{t}] \leq \useconstant{c:large_deviation}^{-1}e^{-\useconstant{c:large_deviation}t},
\end{equation}
for all $t \geq 1$. Markov's inequality concludes the proof.
\end{proof}

\section{Decoupling}
~
\par This section contains the proof of Theorem~\ref{teo:decoupling}, the main step towards the proof of Theorem~\ref{t:large_density}. We first prove this theorem with the aid of an auxiliary proposition, whose proof can be found in Subsection~\ref{subsec:coupling}.

\par In the following, we say that a process $\eta=(\eta_{t})_{t \geq 0}$ has density $\rho>0$ if the initial distribution of the process is a product of i.i.d. $\poisson(\rho)$ random variables.
\newconstant{c:coupling}
\newconstant{c:coupling_2}
\par The proof of the decoupling relies on the construction of a coupling between two processes with different densities such that the process with higher density dominates the less dense one inside a box with large probability. 
\begin{prop}\label{prop:bigcoupling}
There exists a positive constant $\useconstant{c:coupling} = \useconstant{c:coupling}(d)>0$ such that the following holds. For any time $T \geq \useconstant{c:coupling}$, density $\rho \geq 1$, and box $H=[1,n]^{d} \subset \Z^{d}$, there exists a coupling $\P$ between two processes $\eta=(\eta_{s})_{s \geq 0}$ and $\eta^{*}=(\eta^{*}_{s})_{s \geq 0} $ such that
\begin{enumerate}
\item $\eta$ has density $\rho$ and $\eta^{*}$ has density $\rho^{*}=\rho(1+T^{-\frac{d}{8(d+2)}})$;
\item $\eta$ is independent of $\eta^{*}_{0}$;
\item
\begin{equation}\label{eq:coupling_fail}
\P\left[ \eta_{T}(x) > \eta^{*}_{T}(x), \text{ for some } x \in H \right] \leq \rho^{d+1} (n+T)^{d+1}e^{-\useconstant{c:coupling_2}T^{\frac{1}{12}}},
\end{equation}
where $\useconstant{c:coupling_2}=\useconstant{c:coupling_2}(d)>0$.
\end{enumerate}
\end{prop}

\par We defer the proof of Proposition~\ref{prop:bigcoupling} to Subsection~\ref{subsec:coupling}. For now, we use this proposition to conclude the proof of Theorem~\ref{teo:decoupling}.

\begin{proof}[Proof of Theorem~\ref{teo:decoupling}]
Via a simple change of coordinates, we may assume, without loss of generality, that
\begin{equation}
B_{1}=[a,a+n]^{d} \times [-n,0] \text{ and } B_{2}=[1,n]^{d} \times [T,T+n].
\end{equation}
In particular, $\dist_{V}=T$. In fact, using this coordinates, we can allow $f_{1}$ to depend on the half-plane $\Z^{d} \times (-\infty, 0]$ (see Remark~\ref{remark:extension_decoupling}).

Set $H=[-2n-\rho T, 3n+\rho T]^{d}$ and consider the event
\begin{equation}
A=\left\{\begin{array}{c} \text{some particle that is outside $H$ at time $T$} \\ \text{has a trajectory that intersects $B_{2}$} \end{array} \right\}.
\end{equation}

We now use the coupling from Proposition~\ref{prop:bigcoupling} for $H$ (here we need to choose $\dist_{T}=T$ large enough). Denote by $\P$ the probability measure of the coupling and by $\E$ the expectation with respect to $\P$. We obtain two processes $\eta=(\eta_{s})_{s \geq 0}$ and $\eta^{*}=(\eta_{s}^{*})_{s \geq 0}$ with $\eta$ independent of $\eta^{*}_{0}$ such that
\begin{equation}
\begin{split}
\P\left[ \eta_{T}(x) > \eta^{*}_{T}(x), \text{ for some } x \in H \right] & \leq \rho^{d+1}(2\rho T+5n+T)^{d+1}\exp\left\{-\useconstant{c:coupling_2}T^{\frac{1}{12}}\right\} \\
& \leq \rho^{d+1}(5n+3\rho T)^{d+1}\exp\left\{-\useconstant{c:coupling_2}T^{\frac{1}{12}}\right\} \\
& \leq 5^{d+1}\rho^{2d+2}(n+T)^{d+1}e^{-\useconstant{c:coupling_2}T^{\frac{1}{12}}},
\end{split}
\end{equation}
if $T$ is taken large enough. Notice that above we used the hypothesis that $\rho \geq 1$.

Observe now that, whenever $\eta_{T} \preceq_{H} \eta^{*}_{T}$ (we write $\eta \preceq_{H} \xi$ if $\eta(x) \leq \xi(x)$, for all $x \in H$) and $\eta \notin A$, we have $f_{2}(\eta^{*}) \leq f_{2}(\eta)$. This allows us to estimate
\begin{equation}\label{eq:decoupling_1}
\begin{split}
\E_{\rho^{*}}[f_{1}(\eta^{*})f_{2}(\eta^{*})] & = \E_{\rho^{*}}[f_{1}(\eta^{*})\E\big[f_{2}(\eta^{*})|(\eta^{*}_{s})_{s \leq 0}]\big] \\
& = \E_{\rho^{*}}\big[f_{1}(\eta^{*})\E[f_{2}(\eta^{*})|\eta^{*}_{0}]\big] \\
& \leq \E \Big[f_{1}(\eta^{*})\E[f_{2}(\eta^{*})\charf{\{\eta \notin A\}}\charf{\{\eta_{T} \preceq_{H} \eta^{*}_{T}\}}|\eta^{*}_{0}]\Big]+\P[\eta \in A] \\
& \quad +\P\left[ \eta_{T}(x) > \eta^{*}_{T}(x), \text{ for some } x \in H \right] \\
& \leq \E\big[f_{1}(\eta^{*})\E[f_{2}(\eta)|\eta^{*}_{0}]\big] + \P[\eta \in A] + \rho^{2d+2}(n+T)^{d+1}e^{-\useconstant{c:coupling_2}T^{\frac{1}{12}}} \\
& \leq \E_{\rho^{*}}[f_{1}]\E_{\rho}[f_{2}]+\P[\eta \in A] + \rho^{2d+2}(n+T)^{d+1}e^{-\useconstant{c:coupling_2}T^{\frac{1}{12}}}.
\end{split}
\end{equation}

It remains to estimate the probability of the event $A$. We split this probability according to the position of the particles at time $T$. For $k \geq 1$, let
\begin{equation}
H(k)=\{x \in \Z^{d}: \dist_{\infty}(x,H)=k\},
\end{equation}
and notice that there exists a positive constant $c=c(d)$ such that
\begin{equation}
|H(k)| \leq c(n+\rho T+k)^{d-1}.
\end{equation}
Besides, in order for a particle that is at $H(k)$ at time $T$ to reach $B_{2}$, it needs to perform at least $2n+k+\rho T$ steps. Consider the event
\begin{equation}
A(k)=\left\{\begin{array}{c} \text{some particle that is in $H(k)$ at time $T$} \\ \text{has a trajectory that intersects $B_{2}$} \end{array} \right\}.
\end{equation}
We will bound the probability of $A(k)$ by considering the number of particles in $H(k)$ at time $T$. We obtain, by applying Lemma~\ref{lemma:concentration} twice,
\begin{equation}\label{eq:from_far_away}
\begin{split}
\P_{\rho}[\eta \in A(k)] & \leq \P_{\rho} \left[ \sum_{x \in H(k)} \eta_{T}(x) \geq 2 c \rho (n+\rho T+k)^{d-1}+T+k \right] \\
& \qquad + \left[2 c \rho (n+\rho T+k)^{d-1}+T+k\right]\P\left[ \poisson(n) \geq 2n+k+\rho T \right] \\
& \leq e^{-T-k} + \left[2 c \rho (n+\rho T+k)^{d-1}+T+k\right]e^{-\rho T-k} \\
& \leq c(n^{d}+T^{d}+k^{d})e^{-T-k},
\end{split}
\end{equation}
by further increasing the value of $T$ if necessary.

In particular,
\begin{equation}
\P_{\rho}[\eta \in A] \leq \sum_{k=1}^{\infty}\P_{\rho}[\eta \in A(k)] \leq c(n^{d}+T^{d})e^{-T}.
\end{equation}
Combining the equation above with~\eqref{eq:decoupling_1} concludes the proof.
\end{proof}

\begin{remark}
Using the notation of the proof, notice that we can allow for $f_{1}$ to depend on the whole past $( \eta_{s}^{*} )_{s \leq 0}$. In this case, one only needs to observe that~\eqref{eq:decoupling_1} still remains valid, which follows easily from properties of the conditional expectation. This in particular establishes the extension of Theorem~\ref{teo:decoupling} as stated in Remark~\ref{remark:extension_decoupling}. 
\end{remark}

\subsection{Coupling}\label{subsec:coupling}
~
\par In this subsection we present the proof of Proposition~\ref{prop:bigcoupling}. The proof follows the same general steps from~\cite{bt} and~\cite{bt2}. For this reason, we omit some simple computations.

\par The idea for constructing the coupling is to start with two independent configurations $\eta_{0}$ and $\eta^{*}_{0}$ and evolve them simultaneously in order to obtain the domination at time $T$. We first observe that we can restrict ourselves to a larger box $H^{*}$ around $H$ and assume that all particles that end up inside $H$ at time $T$ never leave $H^{*}$. Now, to obtain the domination, we fix a deterministic sequence of times $(s_{i})_{i \geq 0}$ and, in each of these times, we construct a pairing between particles of $\eta$ and of $\eta^{*}$ that are inside $H^{*}$. The evolution is then set in a way that, if a pair of matched particles meets, they continue evolving together. In particular, the probability that there is no domination at time $T$ is bounded by the probability that there exists a particle of $\eta$ that never meets a pair. This will be easily bounded with the aid of Proposition~\ref{prop:coupling}.

\par We now proceed to prove Proposition~\ref{prop:bigcoupling}.

\begin{proof}[Proof of Proposition~\ref{prop:bigcoupling}]
Fix $T$ large enough so that Proposition~\ref{prop:coupling} applies. The coupling will use independent initial configurations $\eta_{0}$ and $\eta^{*}_{0}$ with respective densities $\rho$ and $\rho^{*}$. Besides, we consider two independent copies $\mathcal{S}$ and $\mathcal{S}'$ of the graphical construction presented in Section~\ref{sec:model}.

The process $\eta=(\eta_{t})_{t \geq 0}$ will follow the walks from $\mathcal{S}$, while the process $\eta^{*}=(\eta^{*}_{s})_{x \geq 0}$ will alternate between the two constructions. This implies that $\eta$ is independent of $\eta^{*}_{0}$.

Consider the sequence of times $s_{k}=k T^{\frac{1}{d+2}}$, for $k=0, \dots \lfloor T^{\frac{d+1}{d+2}} \rfloor$, and fix the set
\begin{equation}
H^{*}= \Big[-3 \rho T, n+3 \rho T \Big]^{d}.
\end{equation}
Proceeding as in~\eqref{eq:from_far_away}, we can bound the probability of the event
\begin{equation}
A=\left\{\begin{array}{c} \text{ there exists a particle of } \eta \\ \text{that is outside } H^{*} \text{ for some time } s_{k} \\ \text{and is inside } H \text{ at time } T \end{array} \right\}
\end{equation}
by
\begin{equation}
\P\left[A\right] \leq cT^{\frac{d+1}{d+2}}(n^{d}+T^{d})e^{-T} \leq cT^{\frac{d+1}{d+2}}(n^{d}+T^{d})e^{-\rho T}.
\end{equation}

Let $L= \left\lfloor \frac{\useconstant{c:brw_space}}{2\sqrt{d}} T^{\frac{1}{2(d+2)}} \right\rfloor$, and, for $i \in \Z^{d}$, write
\begin{equation}
H(i) = iL+[0,L)^{d}.
\end{equation}
Set $I= \{i \in \Z^{d}: H(i) \cap H^{*} \neq \emptyset\}$, and notice that
\begin{equation}
|I| \leq \left(\frac{\roof{(6\rho T+n+1)}}{L}\right)^{d} \leq c\rho^{d} \left(n+ T \right)^{d}.
\end{equation}
Observe that $H^{*} \subset \bigcup_{i \in I} H(i)$. If necessary, we increase $H^{*}$ to coincide with the union $\bigcup_{i \in I} H(i)$.

We will perform the pairing inside each box $H(i)$ with $i \in I$. Concentration bounds (via exponential Markov inequality together with the inequality $\log(1+x) \leq x-\tfrac{1}{4}x^{2}$, for $x \in [0,1]$) on the number of particles yields, for each $i \in \Z^{d}$,
\begin{equation}\label{eq:coupling_poisson_1}
\P\left[ \sum_{x \in H(i)} \eta_{t}(x) \geq \frac{\rho+\rho^{*}}{2}L^{d}\right] \leq \exp\left\{-\frac{1}{8}T^{-\frac{d}{4(d+2)}}L^{d}\right\},
\end{equation}
and
\begin{equation}\label{eq:coupling_poisson_2}
\P\left[ \sum_{x \in H(i)} \eta^{*}_{t}(x) \leq \frac{\rho+\rho^{*}}{2}L^{d}\right] \leq \exp\left\{-\frac{1}{24}T^{-\frac{d}{4(d+2)}}L^{d}\right\},
\end{equation}
In particular, if
\begin{equation}
B_{t} = \left\{ \sum_{x \in H(i)} \eta^{*}_{t}(x) \leq \sum_{x \in H(i)} \eta_{t}(x), \text{ for some } i \in I \right\},
\end{equation}
then we can use~\eqref{eq:coupling_poisson_1} and~\eqref{eq:coupling_poisson_2} to obtain
\begin{equation}\label{eq:coupling_bound_B}
\P\left[ B_{t} \right] \leq 2|I|\exp\left\{-\frac{1}{4}T^{-\frac{d}{4(d+2)}}L^{d}\right\}.
\end{equation}
Here it is important to notice that the bound above does not require $\eta_{t}$ and $\eta^{*}_{t}$ to be independent, only that they have the correct marginal distributions.

We now construct the evolution of the process $\eta^{*}=(\eta^{*}_{s})_{s \geq 0}$. If we are in the event $A \cup B_{0}$, then $\eta^{*}$ evolves using the graphical construction given by the paths $\mathcal{S}'$ and, consequently, $\eta$ and $\eta^{*}$ have independent evolutions. Assume we are in $A^{c} \cap B_{0}^{c}$. We perform a pairing between particles of $\eta_{0}$ and $\eta^{*}_{0}$ inside each set $H(i)$, for $i \in I$. This pairing is deterministic and follows the following steps.
\begin{enumerate}
\item First pair as many particles as possible of $\eta_{0}$ to particles of $\eta^{*}_{0}$ that are in the same site.
\item Pair the remaining particles of $\eta_{0}$ to particles of $\eta^{*}_{0}$ that are in the same sub-box $H(i)$.
\end{enumerate}
Observe that this pairing is always possible in the event $B_{0}^{c}$.

This pairing between particles will be used to construct the evolution of the process $\eta^{*}$. Particles of this process will use the trajectories of the construction $\mathcal{S}'$ until the time they share a site with their corresponding pair. When this happens, the particle will follow the trajectory from $\mathcal{S}$ that its pair from $\eta$ uses. At the times $(s_{k})_{k=1}^{\lfloor T^{\frac{d+1}{d+2}} \rfloor}$, these pairings are remade, following the rules mentioned above (in particular from first step in the construction of the pairings, it is possible to retain pairs of particles that meet before this rearrangement). These rules imply that the number of particles that meet a pair cannot decrease when the pairings are remade.

In view of Proposition~\ref{prop:coupling}, the probability that a particle meets its couple between times $s_{k}$ and $s_{k+1}$ is at least $\useconstant{c:rw_meeting}T^{-\frac{d}{2(d+2)}}>0$. Furthermore, each particle has $\lfloor T^{\frac{d+1}{d+2}} \rfloor \geq \lfloor T^{\frac{2}{3}} \rfloor$ attempts to find a pair. We obtain the bound
\begin{equation}
\begin{split}
\P \left[ \begin{array}{c} \text{a given particle does not find any of its pairs} \\ \text{ in any of its allowed attempts, } A^{c}, \cap_{i=0}^{\lfloor T^{\frac{d+1}{d+2}} \rfloor-1} B_{s_{i}} \end{array} \right] & \leq (1-\useconstant{c:rw_meeting}T^{-\frac{d}{2(d+2)}})^{\lfloor T^{\frac{2}{3}} \rfloor} \\
& \leq e^{-\frac{\useconstant{c:rw_meeting}}{2}T^{\frac{d+8}{6(d+2)}}} \leq e^{-\frac{\useconstant{c:rw_meeting}}{2}T^{\frac{1}{6}}}.
\end{split}
\end{equation}

In particular, we can use union bounds to obtain
\begin{equation*}
\begin{split}
\P & \left[ \eta_{T}(x) > \eta^{*}_{T}(x), \text{ for some } x \in H \right] \leq \P[A]+ \sum_{i=0}^{\lfloor T^{\frac{d-1}{d-2}} \rfloor-1}\P[B_{s_{i}}] \\ 
& \qquad \qquad \qquad + \rho (6 \rho T+n+1)^{d}\P \left[ \begin{array}{c} \text{a given particle does not find any of its pairs} \\ \text{ in any of its allowed attempts, } A^{c}, \cap_{i=1}^{\lfloor T^{\frac{d+1}{d+2}} \rfloor-1} B_{s_{i}} \end{array} \right] \\
& \qquad \qquad \qquad \leq  cT^{\frac{d+1}{d+2}}(n^{d}+T^{d})e^{-T}+ 2T^{\frac{d+1}{d+2}}|I|e^{-\frac{1}{4}T^{-\frac{d}{4(d+2)}}L^{d}}+\rho(6\rho T+n+1)^{d}e^{-\frac{\useconstant{c:rw_meeting}}{2}T^{\frac{1}{6}}} \\
& \qquad \qquad \qquad \leq \rho^{d+1} (n+T)^{d+1}e^{-\useconstant{c:coupling_2}T^{\frac{1}{12}}},
\end{split}
\end{equation*}
concluding the proof.
\end{proof}

\section{The large density case}\label{sec:large_density}
~
\par We now focus on the proof of Theorem~\ref{t:large_density}. Here, we develop a renormalization structure for the infection front. We start by providing elementary bounds for the probability that the environment behaves exceptionally bad, and for events where the infection act abnormally. In Subsection~\ref{subsec:box_notation}, we establish the main notations used for the renormalization scheme presented in Subsection~\ref{subsec:renormalization}. Finally, Subsection~\ref{subsec:proof} contains the proof of Theorem~\ref{t:large_density}.

\subsection{Elementary bounds}
~
\par In this subsection, we present some rough initial estimates that will be used to bound the events where the infection process behaves exceptionally bad. These estimates are not sharp and rely mostly on union bounds and large deviations for the number of particles or jumps in a given time interval.

\newconstant{c:many_particles}

\par Our first lemma bounds the probability that a given vertex has many particles at some moment before a given time.
\begin{lemma}\label{lemma:buzy_site}
There exists a positive constant $\useconstant{c:many_particles}$ such that, for all $L \geq 1$ and density $\rho \leq L^{2}$,
\begin{equation}
\P_{\rho}\left[\begin{array}{c} \eta_{t}(0) \geq L^{d+4}, \\ \text{for some } t \in [0,L] \end{array}\right] \leq \useconstant{c:many_particles}e^{-\useconstant{c:many_particles}^{-1}L}.
\end{equation}
\end{lemma}

\begin{proof}
Notice that, in order for the origin to have many particles before time $L$, it is necessary that a large ball around it starts with many particles, or there exists a particle that performs many jumps before time $L$. Proceeding as in~\eqref{eq:from_far_away}, we can bound
\begin{equation}
\begin{split}
\P_{\rho} \left[\begin{array}{c} \eta_{t}(0) \geq L^{d+4}, \\ \text{for some } t \in [0,L] \end{array}\right] & \leq \P_{\rho}\left[\sum_{x \in B(0,3L)} \eta_{0}(x) \geq L^{d+4}\right] \\
& \quad + \P_{\rho} \left[\begin{array}{c} \text{ there exists some particle that} \\ \text{starts outside } B(0,3L) \\ \text{and reaches the origin before time } L \end{array}\right] \\
& \leq \useconstant{c:many_particles}e^{-\useconstant{c:many_particles}^{-1}L},
\end{split}
\end{equation}
by choosing $\useconstant{c:many_particles}$ appropriately.
\end{proof}

\newconstant{c:leave_box}

\par The next lemma bounds the probability that the infection process travels abnormally fast during a given time interval.
\begin{lemma}\label{lemma:leave_box}
There exists a positive constant $\useconstant{c:leave_box}$ such that, for all $L \geq 2$ and density $\rho \leq L^{2}$,
\begin{equation}
\P_{\rho}\left[\begin{array}{c} \text{there exists an infected particle} \\ \text{outside } [-L^{d+6}, L^{d+6}]^{d} \text{ before time } L \end{array}\right] \leq \useconstant{c:leave_box}e^{-\useconstant{c:leave_box}^{-1}L}.
\end{equation}
\end{lemma}

\begin{proof}
We can bound the probability of the event above by the probability that some vertex inside $[-L^{d+6}, L^{d+6}]^{d}$ has many particles at some moment before time $L$ or the infection process travels fast through a field of vertices that are typical.
Let $A$ denote the event in the statement of the lemma, and define the events
\begin{equation}
B = \left\{\begin{array}{c} \eta_{t}(x) \geq L^{d+4}, \text{ for some} \\ t \in [0,L] \text{ and } x \in [-L^{d+6}, L^{d+6}]^{d}\end{array} \right\},
\end{equation}
and
\begin{equation}
C = \left\{\begin{array}{c} \text{there exists a path  of size } L^{d+6}, 0=x_{0} \sim x_{1} \sim \cdots \sim x_{L^{d+6}} \\ \text{and a sequence of particles } X_0, X_{1}, \dots, X_{L^{d+6}-1} \text{ such that} \\ X_{i} \text{ jumps from } x_{i} \text{ to } x_{i+1} \text{ after } X_{i-1} \text{ arrives at } x_{i} \text{ and before time } L \end{array} \right\}.
\end{equation}

Notice that
\begin{equation}
\P_{\rho}[A] \leq \P_{\rho}[B]+\P_{\rho}[C \cap B^{c}].
\end{equation}
Lemma~\ref{lemma:buzy_site} gives a bound on the probability of $B$. As for the probability of $C \cap B^{c}$, we use a union bound on the possible choices for the path. Notice that, since the number of particles that are at position $x_{i}$ is bounded by $L^{d+4}$, the probability of having particles realize a given fixed path can be bounded by the probability that a Poisson process of intensity $L^{d+4}$ has more than $L^{d+6}$ ticks up to time $L$. This yields the bound
\begin{equation}
\begin{split}
\P_{\rho}[C \cap B^{c}] & \leq (2d)^{L^{d+6}}\P\left[\poisson(L^{d+5}) \geq L^{d+6} \right] \\
& \leq (2d)^{L^{d+6}}e^{-\useconstant{c:concentration}L^{d+6}\log \frac{L^{d+6}}{L^{d+5}}}.
\end{split}
\end{equation}
The proof is completed by combining the bounds above and choosing the constants appropriately.
\end{proof}

\newconstant{c:trigger}

\par Recall that we defined the front of the infection towards the direction $e_{1}$ as
\begin{equation}\label{eq:front}
r_{t}=\sup \{ \langle x, e_{1} \rangle : \xi_{t}(x)>0\}.
\end{equation}
The next lemma states that, provided the density is large enough depending on the time parameter, there is a big probability that $r_{t}$ is large.
\begin{lemma}\label{lemma:trigger}
There exists a positive constant $\useconstant{c:trigger}>0$ such that, for any positive integer $L \geq 1$
\begin{equation}
\P_{L^\frac{1}{2}}[r_{L} < 8 L] \leq \useconstant{c:trigger}Le^{-\useconstant{c:trigger}^{-1}L^{\frac{1}{2}}}.
\end{equation}
\end{lemma}

\begin{proof}
Consider the event $A_{k,i}$ defined as
\begin{equation}
A_{k,i} = \left\{\begin{array}{c} \text{between times } k+\frac{i}{8} \text{ and } k+\frac{i+1}{8}, \text{ there exists a particle at position} \\ (k+i)e_{1} \text{ that jumps to position } (k+i+1)e_{1}, \text{ and there exists particles} \\ \text{at positions } (k+i)e_{1} \text{ and } (k+i+1)e_{1} \text{ that do not jump} \end{array} \right\},
\end{equation}
and notice that, in $\cap_{k=0}^{L-1}\cap_{i=0}^{7}A_{k,i}$, we have $r_{t} \geq 8L$. This is due to the fact that, between times $k$ and $k+1$, in the intersection $\cap_{k=0}^{L-1}\cap_{i=0}^{7}A_{k,i}$ the infection spreads from $8k e_{1}$ to $8(k+1)e_{1}$ and thus, at time $L$, there exists an infected particle at position $8L$.

Furthermore, we can bound the probability
\begin{equation}
\begin{split}
\P_{L^{\frac{1}{2}}}[A_{k,i}^{c}] & \leq \P\left[\poisson((1-e^{-\frac{1}{8}})p(e_{1})L^{\frac{1}{2}}) = 0\right]+2\P\left[\poisson(e^{-\frac{1}{8}}L^{\frac{1}{2}})=0\right] \\
& = e^{-(1-e^{-\frac{1}{8}})p(e_{1})L^{\frac{1}{2}}}+2e^{-e^{-\frac{1}{8}}L^{\frac{1}{2}}}.
\end{split}
\end{equation}
Combining this with union bound yields
\begin{equation}
\P_{L^{\frac{1}{2}}}[r_{L} < 8L] \leq \sum_{k=0}^{L-1}\sum_{i=0}^{7}\P_{L^{\frac{1}{2}}}\left[A_{k,i}^{c}\right] \leq 27Le^{-\useconstant{c:trigger}L^{\frac{1}{2}}},
\end{equation}
for some positive constant $\useconstant{c:trigger}$, and concludes the proof.
\end{proof}

\subsection{The box notation}\label{subsec:box_notation}
~
\par Let us now introduce the notation in order to develop our renormalization analysis. We first fix a sequence of scales by setting
\begin{equation}
L_{0}=L> 2^{d} \quad \text{and} \quad L_{k+1} = L_{k}^{d+7}.
\end{equation}
The initial value $L_{0}$ will be chosen later to be a large enough integer.

\par For each $k \in \N_{0}$, define the space-time box
\begin{equation}
B_{k} = \left[-L_{k}^{d+6}, L_{k}^{d+6} \right]^{d} \times [0,L_{k}],
\end{equation}
and for $m \in \Z^{d} \times L_{k} \N_{0}$, let $B_{k}(m)$ denote the translated box $B_{k}(m) = m+B_{k}$.

\par Define also the sequence of velocities
\begin{equation}
\vartheta_{0}=8 \quad \text{and} \quad \vartheta_{k+1} = \vartheta_{k}-\frac{6}{\pi^{2}(k+1)^{2}}.
\end{equation}
Notice that $\lim \vartheta_{k} = 7$.

Our goal is to bound the probability that the infection process travels very slowly towards the $e_{1}$ direction. The continuous time nature of the process implies that events of this form do not have a bounded support. For this reason, we introduce events that approximate these and have support contained in the box $B_{k}$. For each $k \in \N_{0}$, define
\begin{equation}\label{eq:boundary}
R_{k} =\left\{(x,t) \in B_{k}: ||x||_{\infty}=L_{k}^{d+6} \text{ or } x_{1} \geq \vartheta_{k}L_{k} \text{ and } t=L_{k}\right\},
\end{equation}
and, for $m \in \Z \times L_{k} \N_{0}$, let $R_{k}(m)$ denote the translated set $R_{k}(m)=m+R_{k}$. Figure~\ref{fig:renorm} contains a representation of the sets $B_{k}$ and $R_{k}$ for $d=1$.

\begin{figure}[h]\label{fig:renorm}
\begin{center}
\begin{tikzpicture}

\fill[black!20!white](0,0) .. controls (0.7,0.2) and (0.7,0.3) .. (1,0.6)--(1,0.6) .. controls (1,0.6) and (2,1.5) .. (0,2)--(-2,2) .. controls (-3,1) and (-2,0.2) .. (-1,0);
\draw[<->, thick] (-4,0)--(4,0);
\draw[thick](-3,0)--(-3,2)--(3,2)--(3,0);
\draw[line width =3](-3,0)--(-3,2);
\draw[line width =3](1.5,2)--(3,2)--(3,0);
\draw[->, thick](0,0)--(2,8/3);
\node[right] at (2,2.5){$x=\vartheta_{k} t$};
\node[right] at (3,1){$R_{k}$};
\node[left] at (-3,1){$R_{k}$};

\end{tikzpicture}
\caption{A representation for the sets $B_{k}$ and $R_{k}$. We represent the event $E_{k}$ by considered the shaded area as the infection process.}
\end{center}
\end{figure}
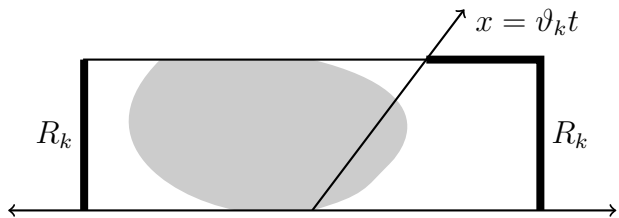

For $m = (x,t) \in \Z \times \N_{0}$, we denote by $(\xi^{m}_{s})_{s \geq t}$ for the infection process starting with initial configuration $\eta_{t}$ and initial collection of infected particles the ones in $x$.

Consider the events
\begin{equation}\label{eq:event_E}
E_{k}(m) = \left\{\xi_{t}^{m}(x)=0, \text{ for all } (x,t) \in R_{k}(m)\right\}.
\end{equation}
See Figure~\ref{fig:renorm} for a representation of the event above for $d=1$. Observe that the events $E_{k}(m)$ are non-increasing and have support inside $B_{k}(m)$. When $m=(0,0)$, we omit this in the definition and denote $E_{k}(0,0)$ simply by $E_{k}$.

\par We introduce the sequence of densities
\begin{equation}
\rho_{0}= \sqrt{L_{0}} \quad \text{ and } \quad \rho_{k+1}=\rho_{k}(1+L_{k}^{-\sfrac{d}{8(d+2)}}),
\end{equation}
and notice that, for each $k$, we have $\rho_{k+1} \leq L_{k}^{2}$, since $L_{0} \geq 2$. Furthermore the sequence $(\rho_{k})_{k \in \N_{0}}$ is monotone increasing and $\rho_{\infty} = \lim \rho_{k}$ exists and is finite.
For each $k$, define also
\begin{equation}
p_k = \P_{\rho_{k}}\left[ E_{k}(m) \right],
\end{equation}
which does not depend on $m$, by translation invariance of the process.

\par Finally, let $M_{k}$ denote the set of indices $m$ in scale $k$ in the box $B_{k+1}$, namely
\begin{equation}
M_{k} = \left(\Z^d \times L_{k} \N_{0}\right) \cap B_{k+1},
\end{equation}
and observe that
\begin{equation}
|M_{k}| \leq L_{k+1}^{6d+1}.
\end{equation}

\subsection{Estimates on $p_{k}$}\label{subsec:renormalization}
~
\par In this subsection, we provide estimates on $p_{k}$, by proving that, provided $L_{0}$ is large enough, the sequence $(p_{k})_{k \in \N_{0}}$ decays quickly.

\newconstant{c:recurssion}

\par The first lemma we prove here states a recursive inequality that relates $p_{k}$ to $p_{k+1}$.
\begin{lemma}\label{lemma:renorm_1}
There exist positive constants $\ell_{0}$ and $A$ such that, if $L_{0} \geq \ell_{0}$ is an integer, then, for all $k \in \N_{0}$,
\begin{equation}
p_{k+1} \leq L_{k+1}^{A}\left(p_{k}^{d+8}+e^{-\useconstant{c:recurssion}L_{k}^{\frac{1}{12}}} \right).
\end{equation}
\end{lemma}

\begin{proof}
Choose $\ell_{0}$ is large enough such that, if $L_{0} \geq \ell_{0}$, then
\begin{equation}
L_{k} > \frac{8\pi^{2}}{3}(2d+14)(k+1)^{2}, \text{ for all } k \in \N_{0}.
\end{equation}

Consider, for each $m =(y,jL_{k}) \in M_{k}$,
\begin{equation}\label{eq:D_k}
D_{k}(m) = D_{k}(y,jL_{k}) = \left\{ \begin{array}{c} \xi_{t}^{m}(x)>0, \text{ for some } (x,t) \text{ with} \\ ||x-y||_{\infty}=L_{k}^{d+6}, (t-jL_{k}) \leq L_{k} \end{array} \right\},
\end{equation}
the event where the infection $\xi^{m}$ travels abnormally fast before time $L_{k}$.

We first claim that, on $E_{k+1}$, one of the following two conditions hold:
\begin{enumerate}
\item For some $m \in M_{k}$, the event $D_{k}(m)$ holds;
\item There are $2d+15$ values $(m_{i})_{i=1}^{2d+15}$ in $M_{k}$ with different time coordinates such that $E_{k}(m_{i})$ holds for all $i \leq 2d+15$.
\end{enumerate}

In order to verify the claim, suppose we are in the event $E_{k+1}$, that the first condition does not hold, and that the second one holds for at most $2d+14$ values of $m$ with different time coordinates.

For each $m =(x,t) \in M_{k}$ such that $\eta_{t}(x)>0$, let $X(m)$ denote a point in $\Z^{d}$ such that
\begin{equation}
\langle X(m), e_{1} \rangle = \max\left\{\langle x, e_{1} \rangle: \xi_{L_{k}}^{m}(x)>0 \right\}.
\end{equation}
By assuming that the first condition above does not hold, we have
\begin{equation}
-L_{k}^{d+6} \leq \langle X(m)-m, e_{1} \rangle \leq L_{k}^{d+6}.
\end{equation}
Furthermore, using that the second condition does not hold, the lower bound can be improved to
\begin{equation}
\langle X(m)-m, e_{1} \rangle \geq \vartheta_{k}L_{k}
\end{equation}
for all but at most $2d+14$ different time coordinates of $m$.

Define the sequence of sites in $\Z^{d}$ as
\begin{equation}
Y_{0}=0 \quad \text{and} \quad Y_{\ell+1}=X((Y_{\ell}, \ell L_{k})).
\end{equation}
Notice that
\begin{equation}
\xi_{L_{k+1}}\left(Y_{\frac{L_{k+1}}{L_{k}}}\right)>0.
\end{equation}

We now estimate
\begin{equation}\label{eq:concatenation}
\begin{split}
r_{L_{k+1}} & \geq \langle Y_{\frac{L_{k+1}}{L_{k}}}, e_{1} \rangle = \sum_{\ell=1}^{\frac{L_{k+1}}{L_{k}}} \langle Y_{\ell}-Y_{\ell-1}, e_{1} \rangle \\
& \geq \vartheta_{k}L_{k}\left(\frac{L_{k+1}}{L_{k}}-2d-14\right) -(2d+14)L_{k}^{d+6} \\
& \geq \vartheta_{k+1}L_{k+1}+\frac{6}{\pi^{2}(k+1)^{2}}L_{k+1}-(2d+14)(\vartheta_{k}L_{k}+L_{k}^{d+6}) \\
& \geq \vartheta_{k+1}L_{k+1}+\frac{6}{\pi^{2}(k+1)^{2}}L_{k+1}-8(2d+14)(L_{k}+L_{k}^{d+6}) \\
& > \vartheta_{k+1}L_{k+1},
\end{split}
\end{equation}
which is a contradiction with the fact that we are in the event $E_{k+1}$.

We just concluded that, on $E_{k+1}$, there are two possible outcomes: either $D_{k}(m)$ holds for some $m \in M_{k}$ or there exists a choice of indices $(m_{i})_{i=1}^{2d+15}$ in $M_{k}$ with different time coordinates such that $E_{k}(m_{i})$ holds for all $i \leq 2d+15$.

Suppose we are in the last case described above, and fix one possible choice of indices $(m_{i}=(x_{i},s_{i}))_{i=1}^{2d+15}$. Observe that, for all $i$, we have $L_{k} \leq s_{i+2}-s_{i} \leq L_{k+1}$.

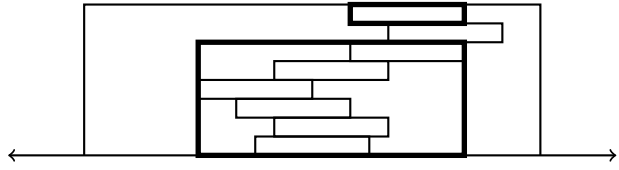
\begin{figure}\label{fig:decoupling}
\begin{center}
\begin{tikzpicture}
\draw[<->, thick] (-4,0)--(4,0);
\draw[thick] (-3,0)--(-3,2)--(3,2)--(3,0);

\draw[line width = 2](0.5,1.75) rectangle (2,2);
\draw[thick](1,1.5) rectangle (2.5,1.75);
\draw[thick](0.5,1.25) rectangle (2,1.5);
\draw[thick](-0.5,1) rectangle (1,1.25);
\draw[thick](-1.5,0.75) rectangle (0,1);
\draw[thick](-1,0.5) rectangle (0.5,0.75);
\draw[thick](-0.5,0.25) rectangle (1,0.5);
\draw[thick](-0.75,0) rectangle (0.75,0.25);
\draw[line width = 2](-1.5,0) rectangle (2,1.5);

\end{tikzpicture}
\caption{The first application of the decoupling estimate. Small boxes represent the supports of the events $E_{k}(m_{i})$ and the bold boxes are the supports of the events considered when applying the decoupling.}
\end{center}
\end{figure}

By possibly further increasing the value of $\ell_{0}$, we can apply Theorem~\ref{teo:decoupling} a few times (see Figure~\ref{fig:decoupling}) to conclude that
\begin{equation}
\begin{split}
\P_{\rho_{k+1}}\left[\bigcap_{i=1}^{2d+15}E_{k}(m_{i}) \right] & \leq \P_{\rho_{k+1}}\left[\bigcap_{i=1}^{2d+13}E_{k}(m_{i}) \right]\P_{\rho_{k}}\left[E_{k}(m_{15}) \right] \\
& \qquad +L_{k}^{4(d+1)}((2L_{k}^{d+6}+1)^{d}+L_{k+1}^{d})e^{-\useconstant{c:decoupling}^{-1}L_{k}^{\frac{1}{12}}} \\
& \leq \P_{\rho_{k}}\left[E_{k} \right]^{d+8}+ (d+8)L_{k}^{4(d+1)}((2L_{k}^{d+6}+1)^{d}+L_{k+1}^{d})e^{-\useconstant{c:decoupling}^{-1}L_{k}^{\frac{1}{12}}}.
\end{split}
\end{equation}
Union bounds, Lemma~\ref{lemma:leave_box} and a change of constants yields
\begin{equation}
\begin{split}
\P_{\rho_{k+1}}[E_{k+1}] & \leq |M_{k}|^{2d+17}\left(\P_{\rho_{k}}\left[E_{k} \right]^{d+8}+ (d+8)\useconstant{c:decoupling}((2L_{k}^{d+6}+1)^{d}+L_{k+1}^{d})e^{-\useconstant{c:decoupling}^{-1}L_{k}^{\frac{1}{12}}}\right) \\
& \qquad + |M_{k}|\P_{\rho_{k+1}}[D_{k}] \\
& \leq L_{k+1}^{A}\left(p_{k}^{d+8}+e^{-\useconstant{c:recurssion}L_{k}^{\frac{1}{12}}} \right),
\end{split}
\end{equation}
concluding the proof.
\end{proof}

\par In the next lemma, we prove that, provided $L_{0}$ is large enough, the probabilities $p_{k}$ decay fast.
\begin{lemma}\label{lemma:renorm_2}
There exist constants $\ell_{1} \geq \ell_{0}$ and $\Delta>0$ such that, if $L_{0} \geq \ell_{1}$, then
\begin{equation}\label{eq:p_k_decay}
p_{k} \leq e^{-\left(\log L_{k}\right)^{1+\Delta}}, \quad \text{for all } k \in \N_{0}.
\end{equation}
\end{lemma}

\begin{proof}
Fix $\Delta$ such that $(d+7)^{1+\Delta} < d+8$. Choose $\ell_{1} \geq \ell_{0}$ large enough such that, for all $L \geq \ell_{0}$,
\begin{equation}
L^{A(d+7)}e^{(d+7)^{1+\Delta}\left( \log L\right)^{1+\Delta}}\left(e^{-(d+8)\left(\log L\right)^{1+\Delta}}+e^{-\useconstant{c:recurssion} L^{\frac{1}{12}}} \right) \leq 1,
\end{equation}
where $A$ and $\useconstant{c:recurssion}$ are given by Lemma~\ref{lemma:renorm_1}.
By further increasing $\ell_{1}$, Lemma~\ref{lemma:trigger} implies
\begin{equation}
p_{0} \leq e^{-\left(\log L_{0}\right)^{1+\Delta}}.
\end{equation}

Assume now that~\eqref{eq:p_k_decay} holds for some value of $k$. Let us verify that it remains true for $k+1$.

By Lemma~\ref{lemma:renorm_1}, we have
\begin{equation}
\begin{split}
e^{\left(\log L_{k+1}\right)^{1+\Delta}}p_{k+1}& \leq  L_{k+1}^{A}e^{\left(\log L_{k+1}\right)^{1+\Delta}}\left(p_{k}^{d+8}+e^{-\useconstant{c:recurssion}L_{k}^{\frac{1}{12}}} \right) \\
& \leq L_{k+1}^{A}e^{(d+7)^{1+\Delta}\left( \log L_{k}\right)^{1+\Delta}}\left(e^{-(d+8)\left(\log L_{k}\right)^{1+\Delta}}+e^{-\useconstant{c:recurssion}L_{k}^{\frac{1}{12}}} \right) \\
& \leq 1,
\end{split}
\end{equation}
concluding the proof.
\end{proof}

\subsection{Proof of Theorem~\ref{t:large_density}}\label{subsec:proof}
~
\par In this subsection, we combine the lemmas provided so far to conclude the proof of Theorem~\ref{t:large_density}.

\par Since the event in~\eqref{eq:travels_slow} is non-increasing, it suffices to verify the statement for one value of $\rho$.

\par Define $\bar{E}_{k}(m)$ analogously to $E_{k}(m)$, but with $R_{k}(m)$ replaced by $\bar{R}_{k}(m)$, an $m$-translation of
\begin{equation}
\bar{R}_{k} =\left\{(x,t) \in B_{k}: ||x||_{\infty}=L_{k}^{d+6} \text{ or } x_{1} \geq 7L_{k} \text{ and } t=L_{k}\right\},
\end{equation}
and notice that this is a non-increasing sequence of events. In particular, for $m=(x,t)$, we have
\begin{equation}
\P_{\rho_{\infty}}[\bar{E}_{k}(m)] \leq \P_{\rho_{\infty}}[E_{k}(m)] \leq p_{k} \leq e^{-\left(\log L_{k}\right)^{1+\Delta}},
\end{equation}
for all $k \geq 0$ and all $m \in M_{k}$, provided $L_{0}$ is large enough. We can now proceed with the proof of Theorem~\ref{t:large_density}.

\begin{proof}[Proof of Theorem~\ref{t:large_density}]
Choose $k_{0}$ large enough such that $\rho_{\infty} \leq L_{k_{0}}^{2}$, and set $t_{0}=L_{k_{0}}$. For $t \geq t_{0}$, let $k \geq k_{0}$ be the only value such that
\begin{equation}
L_{k} \leq t < L_{k+1},
\end{equation}
and choose $\bar{\ell} \in \left[1, \frac{L_{k+1}}{L_{k}}\right]$ such that
\begin{equation}
\bar{\ell}L_{k} \leq t < (\bar{\ell}+1)L_{k}.
\end{equation}
Define the event
\begin{equation}
A_{k} =  \bigcup_{m=(x,t) \in M_{k}} \bar{E}_{k}(m) \cap D_{k}(m),
\end{equation}
where $D_{k}(m)$ is given by~\eqref{eq:D_k}. Notice that
\begin{equation}
\P_{\rho_{\infty}}[A_{k}] \leq L_{k+1}^{6d+1}\left(\P_{\rho_{\infty}}[\bar{E}_{k}] + \P_{\rho_{\infty}}[D_{k}] \right) \leq 2L_{k+1}^{6d+1}e^{-\left(\log L_{k}\right)^{1+\Delta}}.
\end{equation}

On $A_{k}^{c}$, by applying the same concatenation argument of~\eqref{eq:concatenation}, we can easily obtain that
\begin{equation}
r_{\ell L_{k}} \geq 7\ell L_{k}, \quad \text{for all } 1 \leq \ell \leq \frac{L_{k+1}}{L_{k}}.
\end{equation}
In particular, if the above holds simultaneously with $r_{t} \leq t$, then all particles that are at a position that realizes $r_{\bar{\ell}L_{k}}$ at time $\bar{\ell}L_{k}$ must jump at least $7\bar{\ell}L_{k}-t \geq 6L_{k}$ times between times $\bar{\ell}L_{k}$ and $t$. This can be easily bound with concentration on the number of jumps of any given particle. In conclusion, we obtain
\begin{equation}
\begin{split}
\P_{\rho_{\infty}}[r_{t} <t, \eta_{0}(0)>0] & \leq \P_{\rho_{\infty}}[A_{k}]+\P\left[ \poisson(L_{k}) \geq 6K_{k} \right] \\
& \leq 2L_{k+1}^{6d+1}e^{-\left(\log L_{k}\right)^{1+\Delta}}+e^{-L_{k}} \\
& \leq 2t^{(d+7)(6d+1)}\exp\left\{-\left(\log t^{\frac{1}{d+7}}\right)^{1+\Delta}\right\}+e^{-t^{\frac{1}{d+7}}},
\end{split}
\end{equation}
which concludes the proof by possibly changing constants.
\end{proof}

\appendix

\section{Concentration}\label{app:concentration}
~
\par We collect here some basic facts about concentration of Poisson random variables and biased random walks we use in the text.

\newconstant{c:concentration}

\begin{prop}\label{prop:concentration}
There exists $\useconstant{c:concentration}>0$ such that, for any $\rho>0$ and $A \geq 2\rho$ integer, if $X \sim \poisson(\rho)$ then
\begin{equation}
\P\left[X \geq A \right] \leq \exp\left\{-\useconstant{c:concentration} A\log \frac{A}{\rho}\right\}.
\end{equation}
\end{prop}

\begin{remark}\label{remark:concentration}
In particular we obtain the bound
\begin{equation}
\P\left[X \geq A \right] \leq \exp\left\{-\useconstant{c:concentration} A \charf{\{A \geq 10\rho\}} \log10\right\},
\end{equation}
for any positive integer $A \geq 0$.
\end{remark}

\begin{proof}
Use Markov's inequality with $\lambda=\log \frac{A}{\rho}$ to obtain
\begin{equation}\label{eq:exponential_markov}
\begin{split}
\P\left[X \geq A \right] & = \P\left[e^{\lambda X} \geq e^{\lambda A} \right] \\
& \leq \exp\left\{ -\lambda A + \rho \left(e^{\lambda}-1\right)\right\} \\
& \leq \exp\left\{ - A \log \frac{A}{\rho} + \rho \left(\frac{A}{\rho}-1 \right)\right\} \\
& \leq \exp\left\{ - \useconstant{c:concentration} A \log \frac{A}{\rho}\right\}.
\end{split}
\end{equation}
\end{proof}

The following lemma provides another concentration bound for Poisson random variables.
\begin{lemma}\label{lemma:concentration}
For $\rho>0$, if $X \sim \poisson(\rho)$ then
\begin{equation}
\P\left[X \geq A \right] \leq e^{2\rho}e^{-A},
\end{equation}
for any $A>0$.
\end{lemma}

\begin{proof}
Proceeding as in~\eqref{eq:exponential_markov}, we have
\begin{equation}
\P\left[X \geq A \right] = \P\left[ e^{X} \geq e^{A} \right] \leq \exp\left\{ - A + \rho (e-1)\right\} \leq e^{2\rho}e^{-A},
\end{equation}
concluding the proof.
\end{proof}

\newconstant{c:rw_deviation}

\par Our next lemma regards linear deviations of the biased random walk from its mean.
\begin{lemma}\label{lemma:rw_deviation}
For any $\epsilon >0$, there exists a positive constant $\useconstant{c:rw_deviation}=\useconstant{c:rw_deviation}(p(\cdot), d, \epsilon)$ such that, for any $u \geq 0$,
\begin{equation}
\P\left[ ||X_{t}-\vec{v}t|| \geq \epsilon u, \text{ for some } t \in [0,u]\right] \leq \useconstant{c:rw_deviation}^{-1}e^{-\useconstant{c:rw_deviation} u}.
\end{equation}
\end{lemma}

\begin{proof}
First notice that union bound allows us to write
\begin{equation}\label{eq:deviation_1}
\P\left[ ||X_{t}-\vec{v}t|| \geq \epsilon u, \text{ for some } t \in [0,u]\right] \leq \sum_{i=1}^{d}\P\left[ |X_{t}^{i}-v_{i}t| \geq \frac{\epsilon}{d} u, \text{ for some } t \in [0,u]\right],
\end{equation}
where $\left(X_{t}^{i}\right)_{t \geq 0}$ is a continuous time random walk in $\Z$ that jumps to the right with rate $p(e_{i})$ and to the left with rate $p(-e_{i})$, and $v_{i}=p(e_{i})-p(-e_{i})$. Alternatively, we can construct $X^{i}$ from a standard biased random walk $(Y_{t})_{t \geq 0}$ that jumps with rate one and with distribution $q(1)=1-q(-1)=\frac{p(e_{i})}{p(e_{i})+p(-e_{i})}$. Even though the values $q(1)$, $q(-1)$ (and, by consequence, $(Y_{t})_{t \geq 0}$) depend on $i \in [d]$, we omit this dependence. With this construction, if $\bar{v}_{i}=q(1)-q(-1)$, we obtain, for all $i \in [d]$,
\begin{equation}\label{eq:deviation_2}
\begin{split}
\P & \left[ |X_{t}^{i}-v_{i}t| \geq \frac{\epsilon}{d} u, \text{ for some } t \in [0,u]\right] \\
& \qquad \qquad = \P\left[ |Y_{t}-\bar{v}_{i}t| \geq \frac{\epsilon}{d} u, \text{ for some } t \in [0,(p(e_{i})+p(-e_{i}))u]\right].
\end{split}
\end{equation}

We now work with the last quantity. We will prove that
\begin{equation}\label{eq:deviation_3}
\P\left[ |Y_{t}-\bar{v}_{i}t| \geq \epsilon u, \text{ for some } t \in [0,u]\right] \leq c^{-1}e^{-cu},
\end{equation}
for some suitable choice of $c>0$ that depends only on $p(\cdot)$ and $\epsilon>0$. Combining this with~\eqref{eq:deviation_1} and~\eqref{eq:deviation_2} concludes the proof of the lemma.

In order to prove~\eqref{eq:deviation_3}, notice first that $t \mapsto Y_{t}^{i}-\bar{v}_{i}t$ is a continuous time martingale. Hence, Doob's maximal Inequality allows us to conclude that, for any $\lambda>0$,
\begin{equation}\label{eq:deviation_4}
\begin{split}
\P & \left[ |Y_{t}-\bar{v}_{i}t| \geq \epsilon u, \text{ for some } t \in [0,u]\right] \leq \E\left[\exp{\left\{\lambda |Y_{u}-\bar{v}_{i}u|\right\}}\right] e^{-\lambda \epsilon u} \\
& \qquad \qquad \leq e^{-\lambda \epsilon u}\Big(\E\left[\exp{\left\{\lambda (Y_{u}-\bar{v}_{i}u)\right\}}\right]+\E\left[\exp{\left\{-\lambda (Y_{u}-\bar{v}_{i}u)\right\}}\right] \Big).
\end{split}
\end{equation}

By writing $Y_{u}^{i}$ as the sum of a Poissonian number of Bernoulli random variables, one obtains, for every $\lambda \in \R$,
\begin{equation}\label{eq:deviation_5}
\E\left[e^{\lambda Y_{u}}\right] = \exp\left\{u\left(q(1)e^{\lambda}+q(-1)e^{-\lambda} -1\right)\right\}.
\end{equation}

Combining~\eqref{eq:deviation_4} and~\eqref{eq:deviation_5}, we can bound
\begin{equation}
\begin{split}
\P & \left[ |Y_{t}-\bar{v}_{i}t| \geq \epsilon u, \text{ for some } t \in [0,u]\right] \\
& \qquad \qquad \leq e^{-\lambda \epsilon u}\left(\E\left[\exp{\left\{\lambda (Y_{u}-\bar{v}_{i}u)\right\}}\right]+\E\left[\exp{\left\{-\lambda (Y_{u}-\bar{v}_{i}u)\right\}}\right] \right) \\
& \qquad \qquad \leq \exp\left\{u\left(q(1)e^{\lambda}+q(-1)e^{-\lambda}-1-\lambda \bar{v}_{i}-\lambda \epsilon\right)\right\}) \\
& \qquad \qquad \qquad +\exp\left\{u\left(q(1)e^{-\lambda}+q(-1)e^{\lambda}-1+\lambda \bar{v}_{i}-\lambda \epsilon\right)\right\}) \\
& \qquad \qquad \leq 2e^{-cu},
\end{split}
\end{equation}
if $\lambda>0$ is taken small enough, depending of the values of $q(1)$, $q(-1)$ and $\epsilon$. This yields~\eqref{eq:deviation_3} and, together with~\eqref{eq:deviation_1} and~\eqref{eq:deviation_2}, concludes the proof.
\end{proof}

\section{Sampling of random walks}
~
\par In this section, we provide bounds for transition probabilities for biased random walks and collect some consequences of these bounds.

\newconstant{c:zmrw}

\par We first consider zero-mean random walks. The following lemma is a consequence of an analogous result for discrete time balanced random walks.
\begin{lemma}\label{lemma:heat_kernel_zmrw}
Let $(X_{s})_{s \geq 0}$ denote a nearest-neighbor continuous-time random walk with transition probability that has mean zero. There exists a positive constant $\useconstant{c:zmrw}>0$ such that, if all $t \geq 1$ and $x \in \Z^d$ satisfying $||x|| \leq \sqrt{t}$, then
\begin{equation}
\P_{0}[X_{t}=x] \geq \frac{\useconstant{c:zmrw}}{t^{\sfrac{d}{2}}}.
\end{equation}
\end{lemma}

\begin{proof}
First notice that $(X_{s})_{s \geq 0}$ may be realized as a discrete-time zero-mean lazy random walk $(\tilde{X}_{n})_{n \geq0}$ together with a Poisson process $(P_{s})_{s \geq 0}$ in $\R_{+}$ with intensity $2$ that controls the jump times.

According to the remark after Proposition 2.1.2 from~\cite{ll}, there exists a constant $c>0$ such that, if $||x|| \leq \sqrt{n}$, then
\begin{equation}
\P_{x}[\tilde{X}_{n}=x] \geq \frac{c}{n^{\sfrac{d}{2}}}.
\end{equation} 

We now bound
\begin{equation}
\begin{split}
\P_{x}[X_{t}=0] & \geq \sum_{k: \, |k-2t| \leq t} \P[P_{t}=k]\P_{x}[\tilde{X}_{k}=x] \geq \sum_{k: \, |k-2t| \leq t} \P[P_{t}=k]\frac{c}{t^{\sfrac{d}{2}}} \\
& \geq \frac{c}{t^{\sfrac{d}{2}}}\P[|P_{t}-2t| \leq t] \geq \frac{\useconstant{c:zmrw}}{t^{\sfrac{d}{2}}},
\end{split}
\end{equation}
concluding the proof.
\end{proof}

\newconstant{c:brw_time}
\newconstant{c:brw_space}
\newconstant{c:brw_estimate}

\par Our next goal is to obtain an analogous result to Lemma~\ref{lemma:heat_kernel_zmrw} for biased random walks.
\begin{lemma}\label{lemma:heat_kernel_brw}
Let $(X_{s})_{s \geq 0}$ denote a biased nearest-neighbor random walk with transition probability $p(\cdot)$. Assume $p(e_{i})>0$ and $p(-e_{i})>0$, for all $i \in [d]$, and set
\begin{equation}
v = (v_{1}, \dots, v_{d}) = \sum_{y \sim 0} p(y)y \in \R^d.
\end{equation}
There exist positive constants $\useconstant{c:brw_time}$, $\useconstant{c:brw_space}$, and $\useconstant{c:brw_estimate}$ such that, if $t \geq \useconstant{c:brw_time}$ and $x \in \Z^d$ is such that $||x-vt|| \leq \useconstant{c:brw_space}\sqrt{t}$, then
\begin{equation}
\P_{0}[X_{t}=x] \geq \frac{\useconstant{c:brw_estimate}}{t^{\sfrac{d}{2}}}.
\end{equation}
\end{lemma}

\par The proof in based in writing the biased random walk as the sum of drift terms and a zero-mean continuous time random walk, and using the estimate provided by Lemma~\ref{lemma:heat_kernel_zmrw}.
\begin{proof}
For each $i \in [d]$, let
\begin{equation}
p_{i}=\min \{p(e_{i}), p(-e_{i})\},
\end{equation}
and write $Z=2\sum_{i=1}^{d}p_{i}$. Let $(\tilde{X}_{s})_{s \geq 0}$ be a continuous-time random walk with transition probability $q(\cdot)$ given by $q(e_{i})=q(-e_{i})=\frac{p_{i}}{Z}$.

We can write
\begin{equation}
X_{t} = \sum_{i=1}^{d} \sign(v_{i})Y^{i}_{t}e_{i}+\tilde{X}_{t\left(1-\sum_{i=1}^{d}|v_{i}|\right)},
\end{equation}
where $Y^{i}_{t} \sim \poisson(t|v_{i}|)$ are independent.

Split the probability of $X_{t}=x$ according to the value of the sum $\sum_{i=1}^{d} \sign(v_{i})Y^{i}_{t}e_{i}$. This yields
\begin{equation}
\P_{0}[X_{t}=x] = \sum_{y \in \Z^{d}}\P\left[\sum_{i=1}^{d} \sign(v_{i})Y^{i}_{t}e_{i} = y \right] \P_{0}\left[ \tilde{X}_{t\left(1-\sum_{i=1}^{d}|v_{i}|\right)}=x-y\right].
\end{equation}

Set $\useconstant{c:brw_space} = \frac{1}{2}\sqrt{1-\sum_{i=1}^{d}|v_{i}|}$, and observe that, if $||y-vt|| \leq \useconstant{c:brw_space}\sqrt{t}$, then
\begin{equation}
||x-y|| \leq ||x-vt||+||y-vt|| \leq \sqrt{t\left(1-\sum_{i=1}^{d}|v_{i}|\right)},
\end{equation}
so that Lemma~\ref{lemma:heat_kernel_zmrw} implies
\begin{equation}
\begin{split}
\P_{0}[X_{t}=x] & \geq \sum_{y : \, ||y-vt|| \leq \useconstant{c:brw_space}\sqrt{t}}\P\left[\sum_{i=1}^{d} \sign(v_{i})Y^{i}_{t}e_{i} = y \right] \frac{\useconstant{c:zmrw}}{t^{\sfrac{d}{2}}} \\
& \geq \frac{\useconstant{c:zmrw}}{t^{\sfrac{d}{2}}} \P \left[ \left| \left|\sum_{i=1}^{d} \sign(v_{i})Y^{i}_{t}e_{i}-vt \right|\right| \leq \useconstant{c:brw_space}\sqrt{t} \right].
\end{split}
\end{equation}
The central limit theorem implies that, if $t$ is large enough, the last probability above is uniformly bounded from below by some positive constant. This concludes the proof.
\end{proof}

\newconstant{c:rw_meeting}

\par The following proposition states that, provided two random walks do not start very far away, the probability that they meet at a given time $t$ does not decay very fast.
\begin{prop}\label{prop:coupling}
Let $p(\cdot)$ be a transition probability satisfying all hypotheses from Lemma~\ref{lemma:heat_kernel_brw}. There exists a positive constant $\useconstant{c:rw_meeting}>0$ such that, for all $t \geq \useconstant{c:brw_time}$ and $x, y \in \Z^d$ such that $||x-y|| \leq \frac{\useconstant{c:brw_space}}{2}\sqrt{t}$, the following holds. If $(X_{s})_{s \geq 0}$ and $(Y_{s})_{s \geq 0}$ are independent random walks with jump distribution $p(\cdot)$ and initial positions $X_{0}=x$ and $Y_{0}=y$, then
\begin{equation}
\P[X_{t}=Y_{t}] \geq  \frac{\useconstant{c:rw_meeting}}{t^{\sfrac{d}{2}}}.
\end{equation}
\end{prop}

\begin{proof}
Recall Lemma~\ref{lemma:heat_kernel_brw} and fix $\delta>0$ such that
\begin{equation}
\delta \leq \frac{\useconstant{c:brw_estimate}}{t^{\sfrac{d}{2}}} \left|B\left(vt, \frac{\useconstant{c:brw_space}}{2}\sqrt{t} \right) \cap \Z^d \right|,
\end{equation}
where $B(a,r)$ denotes the $L^{\infty}$-ball of $\R^{d}$ with center $a$ and radius $r$. Using that there exists $c>0$ such that $\left|B\left(vt, \frac{\useconstant{c:brw_space}}{2}\sqrt{t} \right) \cap \Z^d \right| \geq c t^{\sfrac{d}{2}}$, for all $t \geq \useconstant{c:brw_time}$, we obtain that $\delta$ can be chosen uniformly positive, for all $t$ large enough.

By Lemma~\ref{lemma:heat_kernel_brw}, we have
\begin{equation}
\P[X_{t}=z] \geq \frac{\useconstant{c:brw_estimate}}{t^{\sfrac{d}{2}}},
\end{equation}
for all $z \in B\left(x+vt, \frac{\useconstant{c:brw_space}}{2}\sqrt{t} \right) \cap \Z^d$. The same holds for the random walk $Y$. From this, we conclude
\begin{equation}
P[X_{t}=Y_{t}] \geq \sum_{z \in B\left(x+vt, \frac{\useconstant{c:brw_space}}{2}\sqrt{t} \right) \cap \Z^d} \P[X_{t}=z] \P[Y_{t}=z] \geq \frac{\useconstant{c:brw_estimate}}{t^{\sfrac{d}{2}}} \frac{\useconstant{c:brw_estimate}}{t^{\sfrac{d}{2}}} \delta \useconstant{c:brw_estimate} t^{\sfrac{d}{2}} = \frac{\useconstant{c:rw_meeting}}{t^{\sfrac{d}{2}}},
\end{equation}
which concludes the proof.
\end{proof}

\bibliographystyle{plain}
\bibliography{mybib}

\begin{thebibliography}{10}

\bibitem{amp}
Oswaldo Alves, Fabio Machado, and Serguei Popov.
\newblock The shape theorem for the frog model.
\newblock {\em The Annals of Applied Probability}, 12(2):533--546, 2002.

\bibitem{ampr}
Oswaldo Alves, Fabio Machado, Serguei Popov, and Krishnamurthi Ravishankar.
\newblock The shape theorem for the frog model with random initial
  configuration.
\newblock {\em Markov Processes Relat. Fields}, 7(4):525--539, 2001.

\bibitem{bt2}
Rangel Baldasso and Augusto Teixeira.
\newblock How can a clairvoyant particle escape the exclusion process?
\newblock {\em Annales de l'Institut Henri Poincar{\'e}, Probabilit{\'e}s et
  Statistiques}, 54(4):2177--2202, 2018.

\bibitem{bt}
Rangel Baldasso and Augusto Teixeira.
\newblock Spread of an infection on the zero range process.
\newblock {\em Annales de l'Institut Henri Poincar{\'e}, Probabilit{\'e}s et
  Statistiques}, 56(3):1898--1928, 2020.

\bibitem{bs}
Itai Benjamini and Alexandre Stauffer.
\newblock Perturbing the hexagonal circle packing: a percolation perspective.
\newblock {\em Annales de l'Institut Henri Poincar{\'e}, Probabilit{\'e}s et
  Statistiques}, 49(4):1141--1157, 2013.

\bibitem{gs2}
Peter Gracar and Alexandre Stauffer.
\newblock Multi-scale {L}ipschitz percolation of increasing events for
  {P}oisson random walks.
\newblock {\em The Annals of Applied Probability}, 29(1):376--433, 2019.

\bibitem{gs1}
Peter Gracar and Alexandre Stauffer.
\newblock Random walks in random conductances: Decoupling and spread of
  infection.
\newblock {\em Stochastic Processes and their Applications}, 129(9):3547--3569,
  2019.

\bibitem{rwrw}
Marcelo Hil{\'a}rio, Frank Den~Hollander, Vladas Sidoravicius, Renato~Soares
  dos Santos, and Augusto Teixeira.
\newblock Random walk on random walks.
\newblock {\em Electronic Journal of Probability}, 20, 2015.

\bibitem{jmr}
Milton Jara, Gregorio Moreno, and Alejandro~F. Ram\'{i}rez.
\newblock Front propagation in an exclusion one-dimensional reactive dynamics.
\newblock {\em Markov Processes and Related Fields}, 14(2):185--206, 2008.

\bibitem{ks}
Harry Kesten and Vladas Sidoravicius.
\newblock The spread of a rumor or infection in a moving population.
\newblock {\em The Annals of Probability}, pages 2402--2462, 2005.

\bibitem{ks3}
Harry Kesten and Vladas Sidoravicius.
\newblock A phase transition in a model for the spread of an infection.
\newblock {\em Illinois J. Math.}, 50(1-4):547--634, 2006.

\bibitem{ks2}
Harry Kesten and Vladas Sidoravicius.
\newblock A shape theorem for the spread of an infection.
\newblock {\em Annals of mathematics}, pages 701--766, 2008.

\bibitem{ll}
Gregory~F Lawler and Vlada Limic.
\newblock {\em Random walk: a modern introduction}, volume 123.
\newblock Cambridge University Press, 2010.

\bibitem{psss}
Yuval Peres, Alistair Sinclair, Perla Sousi, and Alexandre Stauffer.
\newblock Mobile geometric graphs: detection, coverage and percolation.
\newblock {\em Proceedings of the twenty-second annual ACM-SIAM symposium on
  Discrete Algorithms}, pages 412--428, 2011.

\bibitem{pt}
Serguei Popov and Augusto Teixeira.
\newblock Soft local times and decoupling of random interlacements.
\newblock {\em Journal of the European Mathematical Society},
  17(10):2545--2593, 2015.

\bibitem{popov}
Serguei~Yu. Popov.
\newblock Frogs and some other interacting random walks models.
\newblock {\em Discrete Mathematics and Theoretical Computer Science},
  AC:277--288, 2003.

\bibitem{rs}
Alejandro~F. Ram{\'\i}rez and Vladas Sidoravicius.
\newblock Asymptotic behavior of a stochastic combustion growth process.
\newblock {\em Journal of the European Mathematical Society}, 6(3):293--334,
  2004.

\bibitem{ss}
Alistair Sinclair and Alexandre Stauffer.
\newblock Mobile geometric graphs, and detection and communication problems in
  mobile wireless networks, 2010.

\bibitem{stauffer}
Alexandre Stauffer.
\newblock Space-time percolation and detection by mobile nodes.
\newblock {\em The Annals of Applied Probability}, 25(5):2416--2461, 2015.

\bibitem{s}
Alain-Sol Sznitman.
\newblock Vacant set of random interlacements and percolation.
\newblock {\em Annals of mathematics}, pages 2039--2087, 2010.

\end{thebibliography}

\end{document}